\documentclass[10pt]{amsart}
\usepackage{verbatim}

\title[Strong symbolic dynamics for metric Anosov flows]{Strong symbolic dynamics for geodesic flows on CAT(-1) spaces and other metric Anosov flows}

\author{David Constantine}
\address{
Wesleyan University \\
Mathematics and Computer Science Department \\
Middletown, CT 06459}
\email{dconstantine@wesleyan.edu}

\author{Jean-Fran\c{c}ois Lafont}
\address{Department of Mathematics\\
                 Ohio State University\\
                 Columbus, Ohio 43210}
\email{jlafont@math.ohio-state.edu}

\author{Daniel J. Thompson}
\address{Department of Mathematics\\
                 Ohio State University\\
                 Columbus, Ohio 43210}
\email{thompson@math.osu.edu}

\date{\today}

\thanks{J.-F. L. was supported by NSF grants DMS-1510640 and DMS-1812028. D.T. was supported by NSF grant DMS-1461163. }

\newtheorem{thm}{Theorem}[section]
\newtheorem{thmx}{Theorem}

\newtheorem{corx}[thmx]{Corollary}

\newtheorem{lem}[thm]{Lemma}
\newtheorem{prop}[thm]{Proposition}

\theoremstyle{definition}

\newtheorem{defn}[thm]{Definition}

\numberwithin{equation}{section}

\def\Pb{\ifmmode{\Bbb P}\else{$\Bbb P$}\fi}
\def\Z{\ifmmode{\Bbb Z}\else{$\Bbb Z$}\fi}
\def\Q{\ifmmode{\Bbb Q}\else{$\Bbb Q$}\fi}
\def\C{\ifmmode{\Bbb C}\else{$\Bbb C$}\fi}
\def\R{\ifmmode{\Bbb R}\else{$\Bbb R$}\fi}
\def\H{\ifmmode{\Bbb H}\else{$\Bbb H\Bbb N$}\fi}

\def\diam{\operatorname{diam}}

\def\Susp{\operatorname{Susp}}

\def\DDD{\mathcal D}
\def\KKK{\mathcal K}
\def\BBB{\mathcal B}
\def\RRR{\mathcal R}
\def \SSS {\mathcal S}
\def \Proj{\operatorname{Proj}}
\def \Int{\operatorname{Int}}

\def\CAT{\operatorname{CAT}}

\newcommand{\oast}{\ast\kern-.5em{\circ}}
\newcommand{\subast}{\ast\kern-.55em{{}_{\smallsmile}}}
\newcommand{\lbrak}{[\kern-.16em{[}}
\newcommand{\rbrak}{]\kern-.17em{]}}

\usepackage{graphicx} 
\usepackage{amsmath,amssymb,amsthm,amsfonts,amscd,flafter,epsf,bm}
\usepackage{mathtools}
\usepackage{graphics}
\usepackage{MnSymbol}
\usepackage{epsfig}
\usepackage{psfrag}
\usepackage{color}
\usepackage{tikz}
\usetikzlibrary{patterns}
\input xy
\xyoption{all}

\begin{document}

\maketitle
\begin{abstract}
We prove that the geodesic flow on a locally $\CAT(-1)$ metric space which is compact, or more generally convex cocompact with non-elementary fundamental group, can be coded by a suspension flow over an irreducible shift of finite type with H\"older roof function. This is achieved by showing that the geodesic flow is a metric Anosov flow, and obtaining H\"older regularity of return times for a special class of geometrically constructed local cross-sections to the flow. We obtain a number of strong results on the dynamics of the flow with respect to equilibrium measures for H\"older potentials. In particular, we prove that the Bowen-Margulis measure is Bernoulli except for the exceptional case that all closed orbit periods are integer multiples of a common constant. We show that our techniques also extend to the geodesic flow associated to a projective Anosov representation \cite{bcls}, which verifies that the full power of symbolic dynamics is available in that setting.

\end{abstract}

\setcounter{secnumdepth}{2}

\setcounter{section}{0}

%

\section{Introduction}

A \emph{metric Anosov flow}, or \emph{Smale flow}, is a topological flow equipped with a continuous bracket operation which abstracts the local product structure of uniform hyperbolic flows. Examples of metric Anosov flows include Anosov flows, H\"older continuous suspension flows over shifts of finite type, and the flows associated to projective Anosov representations studied by Bridgeman, Canary, Labourie and Sambarino \cite{bcls, BC17}.  We say that a flow has a \emph{Markov coding} if there is a finite-to-one surjective semi-conjugacy  $\pi$ with  a suspension flow over a shift of finite type on a finite alphabet. However, for this symbolic description to be useful, it is also required that the roof function and the map $\pi$ can be taken to be H\"older.  For the purposes of this paper, we call this a \emph{strong Markov coding}. Pollicott showed that Bowen's construction of symbolic dynamics for basic sets of Axiom A flows can be extended to the metric Anosov setting \cite{pollicott} to provide a Markov coding. However, no criteria for obtaining a \emph{strong} Markov coding, which is necessary for most dynamical applications, were suggested.  Furthermore, we see no reason that every metric Anosov flow \emph{should} have a strong Markov coding, since questions of H\"older continuity seem to require additional structure on the space and the dynamics.  In this paper, we give a method for obtaining the strong Markov coding for some systems of interest via the metric Anosov flow machinery. 

Our primary motivation for this analysis is to gain a more complete dynamical picture for the geodesic flow on a compact (and, more generally, convex cocompact), locally $\CAT(-\kappa)$ metric space, where $\kappa>0$, which is a generalization of the geodesic flow on a closed (or convex cocompact) Riemannian manifold of negative curvature. In the closed Riemannian case, the geodesic flow is Anosov, so the system has a strong Markov coding by Bowen's results \cite{bowen-symbolic}. In the convex cocompact Riemannian case, the geodesic flow restricted to the non-wandering set is Axiom A, so Bowen's argument still applies. We show that this extends to the $\CAT(-\kappa)$ case. The majority of previous dynamical results in this area are based on analysis of the boundary at infinity via the Patterson--Sullivan construction. This has yielded many results for the Bowen-Margulis measure \cite{Roblin}, and also recently for a class of equilibrium states \cite{BPP16}. A weak form of symbolic dynamics for geodesic flows on $\CAT(-\kappa)$ spaces, or more generally on Gromov hyperbolic spaces, is due to Gromov \cite{gromov}, extending on an approach of Cannon \cite{Cannon}. Full details were provided by Coornaert and Papadopoulos \cite{cp}. This approach uses topological arguments to give an orbit semi-equivalence with a  suspension over a subshift of finite type.  A priori, orbit semi-equivalence is too weak a relationship to preserve any interesting dynamical properties \cite{GM, KT17}, and it is not known how to improve this construction of symbolic dynamics to a semi-conjugacy. In \cite{CLT}, we used this weak symbolic description to prove that these geodesic flows are expansive flows with the weak specification property, and explored the consequences of this characterization. However, neither the boundary at infinity techniques, nor techniques based on the specification property are known to produce finer dynamical results such as the Bernoulli property. Once the strong Markov coding is established, a treasure trove of results from the literature can be applied. We collect some of these results as applied to geodesic flow for $\CAT(-1)$ spaces as Corollary \ref{t.applications}. The Bernoulli property in particular is an application that is out of reach of the previous techniques available in this setting.

Our first step is to formulate verifiable criteria for a metric Anosov flow to admit a strong Markov coding. In the following statement, the \emph{pre-Markov proper families at scale $\alpha$}, which are formally introduced in Definition \ref{defn:pollicott}, are families of sections to the flow $\BBB=\{B_i\}$, $\DDD=\{D_i\}$, $B_i \subset D_i$  with finite cardinality and diameter less than $\alpha$ satisfying certain nice basic topological and dynamical properties. These families were originally introduced by Bowen and are the starting point for his construction of symbolic dynamics for flows.
\begin{thmx} \label{t.usefulversion}
Let $(\phi_t)$ be a H\"older continuous metric Anosov flow without fixed points. Suppose that there exists a pre-Markov proper family $(\BBB, \DDD)$  satisfying:
\begin{enumerate}
	\item the return time function $r(y)$ for $\BBB$ is H\"older where it is continuous (i.e., on each $B_i \cap H^{-1}(B_j)$ where $H$ is the Poincar\'e first-return map);

	\item the projection maps along the flow $\Proj_{B_{i}}: B_i \times [-\alpha, \alpha] \to B_i$ are H\"older, where $\alpha>0$ is the scale for the pre-Markov proper family. 
\end{enumerate}
Then the flow has a strong Markov coding. 
\end{thmx}
Metric Anosov flows are expansive (see \S \ref{sec:smale}), which implies that their fixed points are a finite set of isolated points, and can thus be removed. Metric Anosov flows satisfy Smale's spectral decomposition theorem \cite{pollicott}. That is, the non-wandering set for $(\phi_t)$ decomposes into finitely many disjoint closed invariant sets on each of which the flow is transitive. In particular, if $(\phi_t)$ is a transitive metric Anosov flow, there are no fixed points and the shift of finite type in the strong Markov coding is irreducible. We verify the criteria of Theorem \ref{t.usefulversion} in our setting, obtaining the following application, which is our main result.

\begin{thmx} \label{t.main}
The geodesic flow for a compact  locally $\CAT(-\kappa)$ space with non-elementary fundamental group, $\kappa>0$, has an irreducible strong Markov coding. 
\end{thmx}

We actually prove Theorem \ref{t.main} in the more general setting of convex cocompact locally $\CAT(-\kappa)$ spaces, with the geodesic flow restricted to the non-wandering set. We prove Theorem \ref{t.main} by giving a geometric construction of a `special' pre-Markov proper family $(\BBB, \DDD)$ for the geodesic flow. The sections are defined in terms of Busemann functions, which are well known to be Lipschitz. We then use the regularity of the Busemann functions to establish (1) and (2) of Theorem \ref{t.usefulversion} for the family $(\BBB, \DDD)$, thus establishing Theorem \ref{t.main}. 

Our second main application is to use similar techniques to study the flow associated to a projective Anosov representation, which is another important example of a metric Anosov flow. Again, the key issue is establishing the regularity properties (1) and (2) of Theorem \ref{t.usefulversion}.  We achieve this using similar ideas to the proof of Theorem \ref{t.main}, although there are some additional technicalities since we must find machinery to stand in for the Busemann functions.

\begin{thmx} \label{t.main2}
The geodesic flow for a projective Anosov representation $\rho: \Gamma \to \mathsf{SL}_m(\mathbb{R})$,  where $\Gamma$ is a hyperbolic group, admits a strong Markov coding.
\end{thmx}

This provides a clean self-contained reference for a key step in the paper \cite{bcls}, addressing an issue in how their statement was justified. We emphasize that this issue can be sidestepped in their examples of interest, and thus does not impact their results in a central way. We explain this here.

In \cite{bcls}, the statement of Theorem \ref{t.main2} is justified by showing that the flow is metric Anosov  \cite[Proposition 5.1]{bcls} and then referencing \cite{pollicott} as saying that this implies the existence of strong Markov coding. This claim also appears in the papers \cite{BC17, BCLS18, PS17, aS16} either explicitly or implicitly through the claim that results that are true for Anosov flows are true for metric Anosov flows via \cite{pollicott}.  However, \cite[Theorem 1]{pollicott} only provides a Markov coding with no guarantee of regularity of the roof function or projection map beyond continuity. When the phase space of the geodesic flow of the representation is a closed manifold, for example in the important case of Hitchin representations, the required regularity can be observed easily from smoothness of the flow and by taking smooth discs for sections in the construction of the symbolic dynamics, as Bowen  argued in the Axiom A case. The case of deformation spaces of convex cocompact hyperbolic manifolds is also unproblematic due to the Axiom A structure. 

If $\Gamma$ is not the fundamental group of a convex cocompact negatively curved manifold, we stress that the phase space of the flow need not be a manifold, so new arguments for regularity are needed. In the context of \cite{bcls}, this is only a minor issue given that they show that the flow is H\"older and demonstrate H\"older continuity of the local product structure. Sections with H\"older return maps can likely be constructed based on these facts. However, to carry this out and incorporate it into Pollicott's symbolic dynamics construction needs rigorous justification. Our proof of Theorem \ref{t.main2}  realizes this general philosophy (we do not use or prove H\"older continuity of the local product structure, but our arguments have a similar flavor), and our argument gives a convenient framework and self-contained reference for the regularity of the Markov coding.

The existence of a strong  Markov coding allows one to instantly apply the rich array of results on dynamical and statistical properties from the literature that are proved for the suspension flow, and known to be preserved by the projection $\pi$. We collect some of these results as they apply to our primary example of the geodesic flow for a compact, or convex cocompact, locally $\CAT(-\kappa)$ space.

\begin{corx}\label{t.applications}
Consider the geodesic flow on a compact, or convex cocompact with non-elementary fundamental group, locally $\CAT(-\kappa)$ space, and let $\varphi$  be a H\"older potential function on the space of geodesics (resp. non-wandering geodesics in the convex cocompact case).  Then there exists a unique equilibrium measure $\mu_\varphi$, and it has the following properties.
\begin{enumerate}
\item $\mu_\varphi$ satisfies the Almost Sure Invariance Principle, the Law of the Iterated Logarithm, and the Central Limit Theorem;
	\item The dynamical zeta function is analytic on the region of the complex plane with real part greater than $h$, where $h$ is the entropy of the flow, and has a meromorphic extension to points with  real part greater than $h- \epsilon$.
	\item If the lengths of periodic orbits are not all integer multiples of a single constant then the system is Bernoulli with respect to $\mu_\varphi$;
	\item If the lengths of periodic orbits are all integer multiples of a single constant and the space is compact and geodesically complete, then $\mu_\varphi$ is the product of Lebesgue measure for an interval with a Gibbs measure for an irreducible shift of finite type; the measure in the base is thus Bernoulli if the shift is aperiodic, or Bernoulli times finite rotation otherwise.
	
\end{enumerate}
\end{corx}
The equilibrium measure for $\varphi=0$ is the measure of maximal entropy, which is known in this setting as the  \emph{Bowen-Margulis measure} $\mu_{BM}$. While items (1), (2), and (3) are true for any topologically transitive system with a strong Markov coding, item (4) additionally uses a structure theorem of Ricks in \cite{ricks}, which applies for geodesic flow on compact, geodesically complete locally $\CAT(0)$ spaces. Finally, we note that in our previous work \cite{CLT}, in the compact case, we used a different approach based on the specification property to show that there is a unique equilibrium measure $\mu_\varphi$.  However, those techniques do not give the strong consequences listed above.

The paper is structured as follows. In \S \ref{sec:background}, we establish our definitions and preliminary lemmas. In \S \ref{sec:Markov}, we establish the machinery required to build a strong Markov coding for a metric Anosov flow, and prove Theorem \ref{t.usefulversion}. In \S\ref{sec:good rectangles}, we study geometrically defined sections to the flow, completing the proof of Theorem \ref{t.main}. In \S \ref{sec:twoexamples}, we extend the construction to projective Anosov representations, proving Theorem \ref{t.main2}. In \S \ref{s.applications}, we discuss applications of the strong Markov coding, proving Corollary \ref{t.applications}.

%

\section{Preliminaries}\label{sec:background}

\subsection{Flows and sections} 

We consider a continuous flow $(\phi_t)_{t \in \mathbb R}$ with no fixed points on a compact metric space $(X, d)$.  For a set $D$, and interval $I$, we write
\[ \phi_I D = \{\phi_t x : x \in D, \ t \in I\}. \]

We say that a flow $(\phi_t)$ is \emph{H\"older continuous} if the map from $X \times [0,1] \to X$ given by $(x,t) \to \phi_t(x)$ is H\"older continuous. It follows that every time$-t$ map is H\"older continuous, and the map $t \to \phi_t(x)$ is H\"older continuous for each $x\in X$.

\begin{defn}\label{defn:section}
For a continuous flow $( \phi_t )$ on a metric space $(X,d)$, a \emph{section} is a closed subset $D\subset X$ and a $\xi>0$ so that the map $(z,t) \mapsto \phi_tz$ is a homeomorphism between $D \times [-\xi, \xi] \to X$ and $\phi_{[-\xi,\xi]}D$.
\end{defn}

For a section $D \subset X$,  we write $\Int_\phi  D$ for the interior of $D$ \emph{transverse to the flow}; that is,
\[\Int_\phi D = D \cap \bigcap_{\epsilon>0}(\phi_{(-\epsilon,\epsilon)}D)^\circ\]
where $Y^\circ$ denotes the interior of $Y$ with respect to the topology of $X$.

For any section $D$, there is a well-defined projection map $\Proj_D: \phi_{[-\xi,\xi]}D \to D$ defined by $\Proj_D(\phi_tz) =z$. By definition, the domain of this map contains a nonempty open neighborhood of $X$. In \cite{BW}, a set $D \subset X$ is defined to be a section if $D$ is closed and there exists $\xi>0$ so that $D \cap \phi_{[-\xi,\xi]}x = \{x\}$ for all $x \in D$. It is easily checked that this is equivalent to Definition \ref{defn:section},  see \cite[\S 5]{BW}.

%

\subsection{Shifts of finite type and suspension flow}

Let $\mathcal{A}$ be any finite set. The full, two-sided shift on the alphabet $\mathcal{A}$ is the dynamical system $(\Sigma, \sigma)$ where
\[ \Sigma = \{ \underline x: \mathbb{Z} \to \mathcal{A} \} \ \mbox{ and } \ (\sigma \underline x)_n = \underline x_{n+1}. \]
We equip $\Sigma$ with the metric
\[ d(\underline x, \underline y) := \frac{1}{2^l} \ \mbox{ where } l = \min\{|n|:\underline x_n \neq \underline y_n\}. \]

A \emph{subshift} $Y$ of the full shift is any closed, $\sigma$-invariant subset of $\Sigma$, equipped with the dynamics induced by $\sigma$. We say that $(Y, \sigma)$ is a symbolic system. Given a $\{0,1\}$-valued $d \times d $ transition matrix $A$, where $d$ is the cardinality of $\mathcal{A}$, a ($1$-step) subshift of finite type is defined by
\[
\Sigma_A = \{\underline x \in \Sigma : A_{x_n x_{n+1}}=1 \mbox{ for all } n\in \mathbb{Z}\}.
\]
This is the class of symbolic spaces that appears in this paper. We now recall the suspension flow construction.

\begin{defn}
Given a symbolic system $(Y, \sigma)$ and a positive function $\rho: Y \to (0,\infty)$, we let 
\[
Y^{\rho} = \{ (\underline x,t): \underline x\in Y, 0\leq t \leq \rho(\underline x) \}/((\underline x, \rho(\underline x))\sim(\sigma \underline x,0))\]
and we define the suspension flow locally by $\phi_s(\underline x, t) = (\underline x, t+s).$ This is the \emph{suspension flow over $(Y, \sigma)$ with roof function $\rho$}. We denote the flow $(Y^{\rho}, (\phi_s))$ by $\Susp(Y, \rho)$.
\end{defn}

%

\subsection{$\CAT(-1)$ spaces and geodesic flow}\label{subsec:CAT (-1)}

In this paper, a \emph{geodesic} is defined to be a local isometry from $\mathbb{R}$ to a metric space. Thus, by our definition, a geodesic is parametrized and oriented.

\begin{defn}
In any metric space $(Y,d_Y)$, the \emph{space of geodesics} is
\[ GY:= \{ c:\mathbb{R} \to Y \mbox{ where $c$ is a local isometry} \}.\]
The \emph{geodesic flow} $(g_t)$ on $GY$ is given for $t \in \mathbb R$ by $g_tc(s) = c(s+t)$. 
\end{defn}

A \emph{geodesic metric space} is a metric space in which the distance between any pair of points can be realized by the length of a geodesic segment connecting them. Given a geodesic metric space $(\tilde X,d_{\tilde X})$, and points $x,y,z \in X$, we can form a geodesic triangle $\Delta(x,y,z)$ in $\tilde X$ and a comparison geodesic triangle $\Delta(\bar x, \bar y, \bar z)$ with the same side lengths in $\mathbb{H}^2$. A point $p \in \Delta(x,y,z)$ determines a comparison point $\bar p \in \Delta(\bar x, \bar y, \bar z)$ which lies along the corresponding side of $\Delta(\bar x, \bar y, \bar z)$ at the same distance from the endpoints of that side as $p$. We say that the space $\tilde X$ is $\CAT(-1)$ if for all geodesic triangles $\Delta(x,y,z)$ in $\tilde X$ and all $p,q \in \Delta(x,y,z)$, we have $d_{\tilde X}(p,q)\leq d_{\mathbb{H}^2}(\bar p, \bar q)$. That is, a space is $\CAT(-1)$ if its geodesic triangles are thinner than corresponding triangles in the model space of curvature $-1$. All $\CAT(-1)$ spaces are contractible, see e.g. \cite[Ch.3, Props 28 \& 29]{troyanov}. A $\CAT(-\kappa)$ space is defined analogously: its geodesic triangles are thinner than corresponding triangles in the model space of curvature $-\kappa$. A $\CAT(- \kappa)$ space can be rescaled homothetically to a $\CAT(-1)$ space. Thus, it suffices to consider $\CAT(-1)$ spaces.

In this paper, $(\tilde X, d_{\tilde X})$ will be a $\CAT(-1)$ space, $\Gamma$ will be a discrete group of isometries of $\tilde X$ acting freely and properly discontinuously, and $X=\tilde X/\Gamma$ will be the resulting quotient.  A space $(X,d_X)$ is \emph{locally} $\CAT(-1)$ if every point has a $\CAT(-1)$ neighborhood, and it is easily checked that $X=\tilde X/\Gamma$ satisfies this property. Conversely, the universal cover of a complete locally $\CAT(-1)$ space is (globally) $\CAT(-1)$, see e.g. \cite[Thm II.4.1]{bh}, so every complete locally $\CAT(-1)$ space arises this way. We assume that $\Gamma$ is non-elementary, i.e. $\Gamma$ does not contain $\mathbb{Z}$ as a finite index subgroup.

The \emph{boundary at infinity} of a $\CAT(-1)$ space is the set of equivalence classes of geodesic rays, where two rays $c,d:[0,\infty) \to \tilde X$ are equivalent if they remain a bounded distance apart, i.e., if $d_{\tilde X}(c(t), d(t))$ is bounded in $t$. We denote this boundary by $\partial^\infty \tilde X$. It can be equipped with the cone topology, see e.g. \cite[Chapter II.8]{bh}.  Given a geodesic $c$, we use $c(-\infty)$ and $c(+\infty)$ to denote the points in $\partial^\infty \tilde X$ corresponding to the positive and negative geodesic rays defined by $c$.

Given $\Gamma$, let $\Lambda \subset \partial^\infty \tilde X$ be the limit set of $\Gamma$, i.e., the set of limit points in $\partial^\infty\tilde X$ of $\Gamma\cdot x$ for an arbitrary $x\in \tilde X$. Let $C(\Lambda) \subset \tilde X$ be the convex hull of $\Lambda$; clearly $C(\Lambda)$ is $\Gamma$-invariant. If $C(\Lambda)/\Gamma$ is compact, we say that the action of $\Gamma$ on $\tilde X$ is \emph{convex cocompact}. We also call the space $X=\tilde X/\Gamma$ \emph{convex cocompact}.  If $\Gamma$ already acts cocompactly on $\tilde X$, then $\Lambda=\partial^\infty \tilde X$, $C(\Lambda) = \tilde X$ and $X= C(\Lambda)/\Gamma$. We assume that $X$ is compact, or convex cocompact.

For a $\CAT(-1)$ space $\tilde X$, the space of geodesics $G\tilde X$ can be identified with $[(\partial^\infty\tilde X \times \partial^\infty\tilde X)\setminus \Delta] \times \mathbb{R}$, where $\Delta$ is the diagonal.  We equip $G\tilde X$ with the metric:
\[ d_{G\tilde X}(\tilde c,\tilde c') :=  \int_{-\infty}^\infty d_{\tilde X}(\tilde c(s),\tilde c'(s)) e^{-2|s|}ds.\]
The factor $2$ in the exponent normalizes the metric so that  $d_{G\tilde X}(\tilde c, g_s \tilde c)=s$. The topologies induced on $G\tilde X$ by this metric and on $[(\partial^\infty\tilde X \times \partial^\infty\tilde X)\setminus \Delta] \times \mathbb{R}$ using the cone topology on $\partial^\infty \tilde X$ agree.
We equip $GX$, the space of geodesics in the quotient $X=\tilde X/\Gamma$, with the metric
\[ d_{GX}(c,c') = \inf_{\tilde c,\tilde c'} d_{G\tilde X}(c,c')\]
where the infimum is taken over all lifts $\tilde c, \tilde c'$ of $c$ and $c'$. Since the set of lifts is discrete, the infimum is always achieved.

When $X$ is compact, the space $GX$ is compact, and we study the flow $(g_t)$ on $GX$. When $X$ is convex cocompact, we need to restrict the geodesics we study so that the phase space for our flow is compact. Given a cocompact action of $\Gamma$ on $\tilde X$, let $\hat G\tilde X$ be the set of geodesics with image in $C(\Lambda)$. This set of geodesics is clearly invariant under the geodesic flow and the action of $\Gamma$, and can be identified with $[(\Lambda \times \Lambda)\setminus \Delta]\times \mathbb{R}$. Then $\hat G X = (\hat G\tilde X)/\Gamma$ consists of those geodesics in $X$ which remain in the compact set $C(\Lambda)/\Gamma$. As long as $\Gamma$ is non-elementary, the geodesic flow on $\hat GX$ is transitive. This follows from, for example, \cite[\S8.2]{gromov}.  $\hat GX$ is the non-wandering set for the geodesic flow on $GX$, and it is compact. In the convex-cocompact case, we assume throughout that $\Gamma$ is non-elementary and we study the geodesic flow $(g_t)$ restricted to $\hat G X$. Clearly, if $X$ is compact, $\hat G\tilde X = G\tilde X$ and $\hat GX = GX$. See \cite{tapie} or \cite{mB95} for further background and references on geodesic flow for convex cocompact manifolds.

%
\subsection{Geometric lemmas}

The following lemma has an elementary proof which can be found in \cite[Lemma 2.8]{CLT}.

\begin{lem}\label{lem:X GX bound}
There exists some $L>0$ such that $d_X(c(0), c'(0)) \leq L d_{GX}(c,c')$.
\end{lem}

The following lemma shows that the time-$t$ map of the geodesic flow is Lipschitz.
\begin{lem}\label{lem:Lipschitz}
Fix any $T>0$. Then for any $t\in[0,T]$, and any pairs of geodesics $x,y\in GX$,
\[ d_{GX}(g_t x, g_t y) < e^{2T} d_{GX}(x,y). \]
\end{lem}

\begin{proof}
By definition, for properly chosen lifts,
\[ d_{GX}(x,y) = \int_{-\infty}^\infty d_{\tilde X}(\tilde x(s), \tilde y(s)) e^{-2|s|}ds. \]
As $g_t\tilde x$ and $g_t\tilde y$ are lifts of $g_tx$ and $g_ty$, we compute:
\begin{align}
	d_{GX}(g_tx,g_ty) & \leq \int_{-\infty}^\infty d_{\tilde X}(\tilde x(s+t),\tilde y(s+t))e^{-2|s|}ds \nonumber \\
					& = \int_{-\infty}^\infty d_{\tilde X}(\tilde x(s),\tilde y(s))e^{-2|s-t|}ds \nonumber \\
					& = \int_{-\infty}^\infty d_{\tilde X}(\tilde x(s),\tilde y(s))e^{-2|s|}\cdot \frac{e^{-2|s-t|}}{e^{-2|s|}}ds \nonumber 	
\end{align}
It is easy to check that $\frac{e^{-2|s-t|}}{e^{-2|s|}} \leq e^{2t}$, which completes the proof.
\end{proof}

It follows that the flow $(g_t)$ is Lipschitz, using Lemma \ref{lem:Lipschitz} and the fact that $d_{GX}(g_sx, g_tx) = |s-t|$ for all $x$,  and all $s, t$ with $|s-t|$ sufficiently small.

%

\subsection{Busemann functions and horospheres}

We recall the definitions of Busemann functions and horospheres.

\begin{defn}
Let $\tilde X$ be a $\CAT(-1)$ space, $p\in \tilde X$ and $\xi\in \partial^\infty\tilde X$, and $c$ the geodesic ray from $p$ to $\xi$. The \emph{Busemann function centered at $\xi$ with basepoint $p$} is defined as 
\[B_p(-,\xi): \tilde X \to \mathbb{R},\] 
\[q \mapsto \lim_{t\to\infty} d_{\tilde X}(q,c(t))-t.\]
\end{defn}

It is often convenient for us to use the geodesic ray $c(t)$ itself to specify the Busemann function centered at $c(+\infty)$ with basepoint $c(0)$. Thus, for a given geodesic ray $c(t)$, we say the \emph{Busemann function determined by $c$} is the function
\[
B_c(-):= B_{c(0)} (-, c(+\infty)).
\]

It is an easy exercise to verify that any Busemann function is 1-Lipschitz, and it is a well-known fact that Busemann functions on $\CAT(-1)$ spaces are convex in the sense that for any geodesic $\eta$, $B_p(\eta(t),\xi)$ is a convex function of $t$ (see, e.g. \cite[Prop II.8.22]{bh}). The level sets for $B_p(-,\xi)$ are called \emph{horospheres}.

%

\subsection{Stable and unstable sets for $\CAT(-1)$ spaces}

In a $\CAT(-1)$ space, we define strong stable and unstable sets in $G X$ generalizing the strong stable and unstable manifolds for negatively curved manifolds. See also \cite[\S2.8]{mB95}.

\begin{defn}
Let $p_{GX}:G\tilde X \to GX$ be the natural projection. Given $c\in GX$ with lift $\tilde c\in G\tilde X$, the \emph{strong stable set through $c$} is 
\[ W^{ss}(c) = p_{GX}\{ \tilde c' \in G\tilde X : \tilde c'(\infty) = \tilde c(\infty) \mbox{ and } B_{\tilde c}(\tilde c'(0)) = 0\}. \]
For any $\delta>0$, 
\[ W^{ss}_\delta(c) = p_{GX}\{ \tilde c' \in G\tilde X : \tilde c'(\infty) = \tilde c(\infty), \  B_{\tilde c}(\tilde c'(0)) = 0 \mbox{ and } d_{G\tilde X}(\tilde c,\tilde c')<\delta \}. \] 

The \emph{strong unstable set through $c$} is 
\[W^{uu}(c) = p_{GX}\{ \tilde c' \in GX : \tilde c'(-\infty) = \tilde c(-\infty) \mbox{ and } B_{-\tilde c}(\tilde c'(0)) = 0\}\] and \[ W^{uu}_\delta(c) = p_{GX}\{ \tilde c' \in G\tilde X : \tilde c'(-\infty) = \tilde c(-\infty), \ B_{-\tilde c}(\tilde c'(0)) = 0 \mbox{ and } d_{G\tilde X}(\tilde c,\tilde c')<\delta \},\] where $-\tilde c(t) = \tilde c(-t)$.
\end{defn}

\begin{lem}\label{lem:contract}
There exists a constant $C>1$ so that for sufficiently small $\delta$, 
\begin{enumerate}
\item if $c' \in W^{ss}_\delta(c)$ and $t>0$, $d_{GX}(g_tc, g_tc') \leq C d_{GX}(c, c') e^{-t}$;
\item if $c'' \in W^{uu}_\delta(c)$ and  $t<0$, $d_{GX}(g_tc, g_tc'') \leq C d_{GX}(c, c'') e^{t}$.
\end{enumerate}
\end{lem}

\begin{proof}
We prove this for the stable sets in $\tilde X$. The result in $X$ follows, and the proof for the unstable sets is analogous.  First we note that in $\mathbb{H}^2$,  given $\delta_0>0$, there exists $K>1$ so that if $\bar c, \bar c'$ are two geodesics with $d_{\mathbb{H}^2}(\bar c(t), \bar c'(t) ) < \delta_0$, $\bar c(\infty)= \bar c'(\infty) = \bar \xi$ and $\bar c(0), \bar c'(0)$ on the same horosphere centered at $\bar \xi$, then $d_{\mathbb{H}^2}(\bar c(t), \bar c'(t) ) \leq K e^{-t} d_{\mathbb{H}^2}(\bar c(0), \bar c'(0)) $ for all $t>0$.

Let $c' \in W^{ss}_\delta(c)$ with $\delta$ small enough so that, via Lemma \ref{lem:X GX bound}, $d_X( c(0), c'(0)) <\delta_0 $. In $\tilde X$, consider the ideal triangle $\Delta$ with vertices $ \tilde c(0)$, $\tilde c'(0)$ and $\tilde c(\infty)= \tilde c'(\infty)=\xi$. 
There exists an ideal comparison triangle $\bar\Delta= \Delta(\bar c(0), \bar c'(0), \bar\xi)$ in $\mathbb{H}^2$ satisfying the $\CAT(-1)$ comparison estimate $d_{\tilde X}(\tilde c(t), \tilde c'(t)) \leq d_{\mathbb{H}^2}(\bar c(t), \bar c'(t))$, see \cite[Prop. 4.4.13]{das-simmons-urbanski}. We obtain for all $t\geq 0$,
\begin{equation}\label{eqn:contract}
	 d_{\tilde X}(\tilde c(t), \tilde c'(t)) \leq d_{\mathbb{H}^2}(\bar c(t), \bar c'(t)) \leq K d_{\tilde X}(\tilde c(0),\tilde c'(0)) e^{-t}.
\end{equation}
Now we calculate:
\begin{align}
	d_{G X}(g_t c, g_t c') & \leq e^{-2t} \int_{-\infty}^0 d_{\tilde X}(\tilde c(u),\tilde c'(u)) e^{-2|u|}du \nonumber \\
		& \ \ \ \ \ \ + K e^{-t} \int_{-t}^\infty d_{\tilde X}(\tilde c(s),\tilde c'(s)) e^{-2|s|} ds, \nonumber
\end{align}
by breaking our calculation of $d_{G X}(g_t c, g_t c')$ into integrals over $(-\infty,0)$ and $(0,\infty)$, applying a change of variables to the first integral, and equation \eqref{eqn:contract} to the second. We then have that $d_{G X}(g_t c, g_t c') \leq (1+K)e^{-t} d_{G X}(c,c')$.
\end{proof}

%

\section{Metric Anosov flows}\label{sec:Markov}

In this section, we define metric Anosov flows and prove Theorem \ref{t.usefulversion}. The definition was first given by Pollicott in \cite{pollicott}, generalizing the definition of a hyperbolic flow in \cite{bowen-periodic, bowen-symbolic}, and building on the discrete-time definition of a Smale space due to Ruelle \cite{Ruelle}; see \cite{iP2014} for a detailed exposition in discrete time.

\subsection{Metric Anosov flows}\label{sec:smale}
A continuous flow $(\phi_t)$ on a compact metric space $(Y, d)$ is a  \emph{metric Anosov flow}, also known as a \emph{Smale flow}, if it is equipped with a notion of local product structure. That is, a bracket operation so that the point $\langle x,y\rangle $ is analogous in the uniformly hyperbolic setting to the intersection of the unstable manifold of $x$ with the strong stable manifold of $y$. We give the definition. We follow the presentations of Ruelle and Pollicott and start by emphasizing the topological structure needed. For $\epsilon>0$, let us write 
\[ (Y \times Y)_\epsilon := \{(x,y) \in Y \times Y : d(x,y) < \epsilon\}. \]
Assume there exists a constant $\epsilon>0$ and a continuous map
\[ \langle ~,~\rangle ~: (Y \times Y)_\epsilon \to Y, \]
which satisfies:
\begin{itemize}
	\item[a)] $\langle x,x\rangle =x$
	\item[b)] $\langle \langle x,y\rangle ,z\rangle =\langle x, z\rangle $
	\item[c)] $\langle x, \langle y,z\rangle \rangle =\langle x, z\rangle $.
\end{itemize}
We further assume that if $(\phi_sx,\phi_sy) \in (Y\times Y)_\epsilon$ for all $s\in[0,t]$,
\begin{itemize}
	\item[d)]$\phi_t(\langle x,y \rangle)= \langle \phi_tx, \phi_ty \rangle$.
\end{itemize}
We define the \emph{local strong stable set according to $\langle ~,~\rangle $} to be
\[ V^{ss}_\delta (x) = \{ u \mid u = \langle u,x\rangle  \mathrm{ ~and~ } d(x, u)< \delta \}, \]
and the \emph{local unstable set according to $\langle ~,~\rangle $} to be
\[ V^{u}_\delta (x) = \{ v \mid v = \langle x,v\rangle \mathrm{ ~and~ }  d(x, v)< \delta \}. \]
It can be deduced from the properties of the bracket operation that for small $\delta>0$, the map $\langle ~,~\rangle : V^{ss}_{\delta} (x)  \times V^{u}_{\delta} (x) \to Y$ is a homeomorphism onto an open set in $Y$, see \cite[\S7.1]{Ruelle}.
By decreasing $\epsilon$ if necessary, we may assume that for each $x \in Y$, the map $\langle ~,~\rangle : V^{ss}_{\epsilon/2} (x)  \times V^{u}_{\epsilon/2} (x) \to Y$ is a homeomorphism onto an open set in $Y$. 

So far, the structure required on $\langle~,~\rangle$ is purely topological. For this bracket operation to capture dynamics analogous to that of an Anosov flow, we need to add dynamical assumptions.
We define the \emph{metric} local strong stable and strong unstable sets as follows:
\[ W^{ss}_\delta (x; C, \lambda) = \{ v \in V^{ss}_\delta (x) \mid d(\phi_t x, \phi_t y) \leq Ce^{-\lambda t}d(x,y) \text{ for } t \geq 0\}  \]
\[ W^{uu}_\delta (x; C, \lambda) = \{ v \in V^{u}_\delta (x) \mid d(\phi_{-t} x, \phi_{-t} y) \leq Ce^{-\lambda t}d(x,y)  \text{ for } t \geq 0\}.  \]

\begin{defn}\label{defn:metric anosov}
Let $(\phi_t)$ be a continuous flow on a compact metric space $(Y,d)$ and let $\epsilon>0$ and $\langle~,~\rangle$ be as described above. We say that $(\phi_t)$ is a \emph{metric Anosov flow} if there exist constants $C, \lambda, \delta_0>0$ and a continuous function $v\colon(Y\times Y)_\epsilon \to \mathbb{R}$ such that, writing $W^{ss}_{\delta_0} (x) = W^{ss}_{\delta_0} (x; C, \lambda)$ and $W^{uu}_{\delta_0} (x) = W^{uu}_{\delta_0} (x; C, \lambda)$, for any $(x, y) \in (Y \times Y)_\epsilon$ we have
\[ W^{uu}_{\delta_0}(\phi_{v(x,y)}x) \cap W^{ss}_{\delta_0}(y) = \{\langle x,y\rangle \}.  \]
 Furthermore, $v(x,y)$ is the unique small value of $t$ so that $W^{uu}_{\delta_0}(\phi_{t}x) \cap W^{ss}_{\delta_0}(y)$ is non-empty.
\end{defn}

We can deduce the following basic control on scales:  for each small $\delta>0$, there exists $\gamma\in(0, \epsilon)$ so that if $x,y \in (Y \times Y)_{\gamma}$, then $W^{uu}_{\delta}(\phi_{v(x,y)}x) \cap W^{ss}_{\delta}(y) = \{\langle x,y\rangle\}$, and $v(x,y)<\delta$. This follows from continuity of $\langle ~,~\rangle $ and $v$, and the fact that $\langle x,x\rangle =x$ implies $v(x,x)=0$. The details are similar to \cite[Lemma 1.5]{bowen-periodic}. 

In examples of metric Anosov flows, we can consider the bracket operation $\langle~,~\rangle$ as being defined by the metric strong stable and unstable sets via the equation in Definition \ref{defn:metric anosov}, and check that the basic topological properties of the bracket operation hold as a consequence of being defined this way. Thus, it is probably helpful to think of the families $W^{uu}_\delta$ and $W^{ss}_\delta$ as the basic objects in the definition.

A hyperbolic set  for a  smooth flow is locally maximal if and only if it has local product structure \cite[Theorem 6.2.7]{FH}.  Thus, metric Anosov flows are generalizations of locally maximal hyperbolic sets for smooth flows. In particular, an Axiom A flow restricted to its non-wandering set is metric Anosov. 

Another class of examples of metric Anosov flows is given by suspension flows by a H\"older continuous roof function over a shift of finite type. The metric Anosov flow structure for the constant roof function case is described in \cite{pollicott}. The stables and unstables and bracket operation for the H\"older roof function case can be obtained by using H\"older orbit equivalence with the constant roof function case. The details are similar to the second proof of Theorem 5.1.16 in \cite[\S6.1]{FH}, which shows that a smooth time change of a hyperbolic set is a hyperbolic set.

The following property of metric Anosov flows follows the standard proof that Axiom A flows are expansive.
\begin{thm}\emph{(}\cite[Cor 1.6]{bowen-symbolic}, \cite[Prop 1]{pollicott}\emph{)}
A metric Anosov flow satisfies the expansivity property. 
\end{thm}
There are at most finitely many fixed points for an expansive flow, and they are all isolated. Expansivity is a corollary of the following result, which says that orbits that are close are exponentially close.

\begin{thm}\label{expclose} \emph{(}\cite[Lemma 1.5]{bowen-symbolic}\emph{)}
For a metric Anosov flow, there are constants $C, \lambda>0$ so that for all $\epsilon>0$, there exists $\delta>0$ so that if $x,y \in Y$ and $h: \R \to \R$ is continuous such that  $h(0)=0$ and $d(\phi_t x, \phi_{h(t)} y) < \delta$ for all $t \in [-T, T]$, then $d(x, \phi_v y)< C e^{-\lambda T}\delta$ for some $|v| < \epsilon$.
\end{thm}

Bowen's proof goes through without change in the setting of metric Anosov flows. In the case of geodesic flow on a $\CAT(-1)$ space, this is a well known property of geodesics in negative curvature: it holds for geodesics in $\mathbb{H}^2$ by standard facts from hyperbolic geometry, and this can be propagated to the universal cover of a locally $\CAT(-1)$ space by using two nearby geodesics to form a comparison quadrilateral in $\mathbb{H}^2$. The details of the argument in this case are contained in the proof of Proposition 4.3 of \cite{CLT}.

\begin{thm} \label{t.smale}
For a compact (resp. convex cocompact), locally $\CAT(-1)$ space $X$, the geodesic flow on $Y= GX$ (resp. $Y=\hat GX$) is a metric Anosov flow.
\end{thm}

\begin{proof}
First, we define $\langle \cdot, \cdot \rangle$ for geodesics in $G\tilde X$, and verify its properties there. For $(c,c')\in (G\tilde X \times G\tilde X)_\epsilon$, define $\langle c,c'\rangle $ to be the geodesic $d$ with $d(-\infty) = c(-\infty)$, $d(+\infty) = c'(+\infty)$  and  $B_{c'}(d(0)) = 0$ (see Figure \ref{fig:smale}).

It is easy to verify that $\langle \cdot, \cdot \rangle$ is continuous and satisfies conditions (a), (b), (c), and (d) from \S \ref{sec:smale}. $V_\delta^{ss}(c)$ consists of geodesics $\delta$-close to $c$ which have the same forward endpoint as $\tilde c$ and basepoint on $B_{\tilde c} = 0$, and $V_\delta^u(c)$ is geodesics $\delta$-close to $c$ which have the same backward endpoint as $c$. 

It follows from Lemma \ref{lem:contract} that for a sufficiently large choice of $C$ and $\lambda = 1$, 
\[W^{ss}_\delta(c;C,\lambda) = \{c': c'(+\infty) = c(+\infty), \ B_{c}(c'(0))=0, \mbox{ and } d_{G\tilde X}(c,c')<\delta \};  \]
\[W^{uu}_\delta(c;C,\lambda) = \{c': c'(-\infty) = c(-\infty), \ B_{-c}(c'(0))=0, \mbox{ and } d_{G\tilde X}(c,c')<\delta \}.  \]
We define $v:(G\tilde X\times G\tilde X)_\epsilon \to \mathbb{R}$ by setting $v(c,c')$ to be the negative of the signed distance along the geodesic $d = \langle c,c'\rangle $ from its basepoint to the horocycle $B_{-c}=0$. This is clearly continuous, and it is easily checked that 
\[  W^{uu}_\delta(g_{v(c,c')}c) \cap W^{ss}_\delta(c') = \langle c,c'\rangle \]
and that for all other values of $t$, $W^{uu}_\delta(g_tc) \cap W^{ss}_\delta(c') = \emptyset$.

Recall that $X = \tilde X /\Gamma$, and note that these constructions are clearly $\Gamma$-equivariant. For sufficiently small $\epsilon$, there clearly exists a small enough $\delta_0$ such that the scale $\delta_0$ metric local strong stable and unstable sets descend to $GX$, and the bracket operation $\langle \cdot , \cdot \rangle$ and the map $v$ descend to $(GX \times GX)_\epsilon$. By construction, these operations have all the desired properties for a metric Anosov flow. If $X$ is compact, this completes the proof.

Now we extend the argument to the case that $X$ is convex cocompact. The argument above applies verbatim to define a continuous operation  $\langle \cdot ,\cdot \rangle $ on $(GX \times GX)_\epsilon$ which satisfies conditions (a), (b), (c), and (d) from \S \ref{sec:smale}. To show that we have a metric Anosov flow on the compact metric space $\hat GX$, all that remains to check is that $\langle \cdot ,\cdot \rangle $ can be restricted to $\hat GX$. If $c, c' \in \hat G\tilde X$, then $c(-\infty), c'(+\infty) \in \Lambda$, so by construction the geodesic $d = \langle c, c'\rangle $ has  $d(-\infty), d(+\infty) \in \Lambda$. Thus $d \in \hat GX$.
\end{proof}

\begin{figure}[h]
\begin{tikzpicture}

\draw (0,0) circle (3.5);

\fill[gray!50!white] (0,0) circle(1);

\draw[thick,green!50!black, dashed] (-.55,1.95) circle(1.47);
\draw[thick,green!50!black, dashed] (.55,-1.7) circle(1.7);

\node at (1.2,0.4) {$\tilde U$};

\draw (-1,-3.4) arc (-19.7:19.7:10cm);
\draw (1,3.35) arc (160.3:199.7:10cm);
\draw[thick, blue] (-1,3.35) --(1,-3.35);

\draw [ultra thick, ->] (-.43,.5) -- (-.45,.8) ;
\draw [ultra thick, ->] (.43,0) -- (.43,.3) ;
\draw [blue, ultra thick, ->] (-.15,.5) -- (-.24,.85) ;

\node at (-1,-2.8){$c'$};
\node at (1,2.8){$c$};

\node[red] at (-1,.15) {$v(c,c')$} ;

\node[blue] at (1.36,1){$<c,c'>$};
\draw[blue, ->] (.75,.9) -- (-.05,.7) ;

\node[red, rotate=16] at (-.18,.2) {{\huge$\{$}} ;

\node[green!50!black] at (-2.4,1) {$B_{c'}=0$} ;
\node[green!50!black] at (2.4,-.3) {$B_{-c}=0$} ;

\end{tikzpicture}
\caption{The geometric construction showing that geodesic flow on a CAT(-1) space is a metric Anosov flow.}\label{fig:smale}
\end{figure}
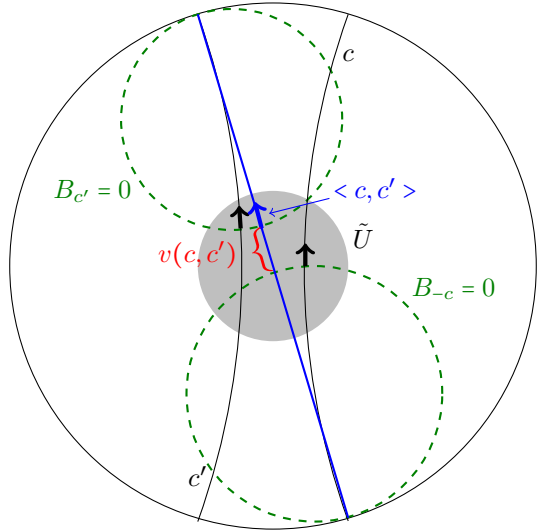

%

\subsection{Sections, proper families, and symbolic dynamics for metric Anosov flows}\label{sec:sections}

We recall the construction of a Markov coding for a metric Anosov flow. We follow the approach originally due to Bowen \cite{bowen-symbolic} for  basic sets for Axiom A flows, which was shown to apply to metric Anosov flows by Pollicott \cite{pollicott}. We recall Bowen's notion of a proper family of sections and a Markov proper family from \cite{bowen-symbolic}.

\begin{defn}\label{defn:proper}
Let $\BBB=\{B_1, \ldots, B_n\}$, and $\DDD=\{D_1, \ldots, D_n\}$ be collections of sections. We say that $(\BBB , \DDD)$ is a \emph{proper family at scale $\alpha>0$} if $\{ (B_i, D_i) : i = 1, 2, \ldots, n \}$ satisfies the following properties:
\begin{enumerate}
	\item $\diam(D_i)< \alpha$ and $B_i \subset D_i$ for each $i \in \{1, 2, \ldots, n\}$;
	\item $\bigcup_{i=1}^n\phi_{(-\alpha, 0)} (\Int_\phi B_i)=Y$;
	\item For all $i\neq j$, if $\phi_{[0, 4\alpha]}(D_i) \cap D_j \neq \emptyset$, then $\phi_{[-4\alpha, 0]}(D_i) \cap D_j = \emptyset$.
\end{enumerate}
\end{defn}

Condition (3) implies that the sets $D_i$ are pairwise disjoint, and the condition is symmetric under reversal of time; that is, it follows that if $\phi_{[-4\alpha, 0]}(D_i) \cap D_j \neq \emptyset$, then $\phi_{[0, 4\alpha]}(D_i) \cap D_j = \emptyset$. In \cite{bowen-symbolic, pollicott}, the time interval in condition (2) is taken to be $[-\alpha, 0]$. Our `open' version of this condition is slightly stronger and convenient for our proofs in \S\ref {sec:good markov}. We now define a special class of proper families, which we call \emph{pre-Markov}.

\begin{defn}\label{def:rectangle}
For a metric Anosov flow, a \emph{rectangle} $R$ in a section $D$ is a subset $R \subseteq \Int_\phi D$ such that for all $x,y \in R$, $\Proj_{D} \langle x,y\rangle  \in R$.
\end{defn}

\begin{defn}[Compare with \S 2 in \cite{pollicott}, \S 7 in \cite{bowen-symbolic}]\label{defn:pollicott}
Let $(\BBB, \DDD)$ be a proper family at scale $\alpha>0$. We say that $(\BBB, \DDD)$ is \emph{pre-Markov} if the sets $B_i$ are closed rectangles and we have the following property:
\begin{equation} \label{premarkov}
\mbox{If }B_i \cap \phi_{[-2 \alpha, 2 \alpha]}B_j \neq \emptyset, \mbox{ then } B_i \subset \phi_{[-3\alpha, 3\alpha]} D_j.
\end{equation}
\end{defn}

The existence of pre-Markov proper families is left as an exercise by both Bowen and Pollicott since it is fairly clear that the conditions asked for are mild; some rigorous details are provided in \cite{BW}. In Proposition \ref{prop:good mpf}, we complete this exercise by providing a detailed proof of the existence of a special class of pre-Markov proper families. For our purposes, we must carry out this argument carefully since it is crucial for obtaining the H\"older return time property of Theorem \ref{t.usefulversion}.

We now define a Markov proper family. This is a proper family where the sections are rectangles, and with a property which can be informally stated as `different forward $\RRR$-transition implies different future, and different backward $\RRR$-transitions implies different past.'

\begin{defn}
A proper family $(\RRR, \SSS)$ is \emph{Markov} if the sets $R_i$ are rectangles, and we have the following Markov property:  let $H$ denote the Poincar\'e return map for $\bigcup_i R_i$ with respect to the flow $(\phi_t)$. Then if $ x \in R_i$ and $ H(x) \in R_j$, and $z \in R_i$ and $H(z) \notin R_j$, then $z \notin V^{ss}_{\diam R_i}(x)$. Similarly, if  $ x \in R_i$ and $ H^{-1}(x) \in R_j$, and $z \in R_i$ and $H^{-1}(z) \notin R_j$, then $z \notin V^u_{\diam R_i}(x)$.
\end{defn}

The reason we call the families defined in Definition \ref{defn:pollicott} \emph{pre-Markov} is because the argument of \S7 of \cite{bowen-symbolic}, and  \S2 of \cite{pollicott} gives a construction to build Markov families out of pre-Markov families.  The motivation for setting things up this way is that the existence of pre-Markov families can be seen to be unproblematic, whereas the existence of proper families with the Markov property is certainly non-trivial. More formally, we have:

\begin{lem}  \label{PtoB}
If $(\BBB, \DDD)$ is a pre-Markov proper family at scale $\alpha$ for a metric Anosov flow, then there exists a Markov proper family $(\RRR, \SSS)$ at scale $\alpha$ so that for all $i$, there exists an integer $j$ and a time  $u_j$ with $|u_j|<<\alpha$ such that $R_i \subset \phi_{u_j}B_j$.
\end{lem}

This is proved in \cite[ \S7]{bowen-symbolic} in the case of Axiom A flows, and the construction in \S2 of \cite{pollicott} adapts this proof to the case of metric Anosov flows, culminating in the statement of \cite[ \S 2.2 `Key Lemma']{pollicott}.  We note that for metric Anosov flows, pre-Markov proper families $(\BBB, \DDD)$ can be found at any given scale $\alpha>0$, and thus Lemma \ref{PtoB} provides Markov proper families at any small scale $\alpha>0$ (in the sense of Definition \ref{defn:proper}).


We note that in \cite{bowen-symbolic, pollicott}, pre-Markov families are also equipped with an `intermediate' family of sections $\KKK$, which are a collection of closed rectangles $K_i \subset  \Int_\phi B_i$, and a scale $\delta>0$ chosen so any closed ball $\overline B(x, 6\delta)$ is contained in some $\phi_{[-2 \alpha, 2 \alpha]}K_i$. Given a pre-Markov proper family, such a collection $\KKK$ and such a $\delta>0$ can always be found. The only role of theses intermediate families is internal to the proof of Lemma \ref{PtoB}, and thus we consider the existence of the family $\KKK$ to be a step in the proof of Lemma \ref{PtoB} rather than an essential ingredient which needs to be included in the definition of proper families. 
 
The proof of Lemma \ref{PtoB}  involves cutting up sections from the pre-Markov family into smaller pieces; this can be carried out so that the resulting sections all have diameter less than $\alpha$. The flow times $u_i$ are used to push rectangles along the flow direction a small amount to ensure disjointness. These times can be taken arbitrarily small, in particular, much smaller than $\alpha$. Note that if $\BBB = \{B_1, \ldots, B_n\}$, then the collection $\RRR= \{R_1, \ldots R_N\}$ provided by Lemma \ref{PtoB} satisfies $N >> n$.

%

\subsection{Markov partitions}

Given a collection of sections $\RRR$, let $H: \bigcup_{i=1}^N R_i \to \bigcup_{i=1}^N R_i$ be the Poincar\'e (return) map, and let $r: \bigcup_{i=1}^N R_i \to (0, \infty)$ be the return time function, which are well defined in our setting.

\begin{defn}
For a Markov proper family $(\RRR, \SSS)$ for a metric Anosov flow, we define the coding space to be
\[ \Sigma=\Sigma(\RRR)= \left \{ \underline x \in \prod_{-\infty}^{\infty}\{1, 2, \ldots, N\} \mid \text{ for all } l, k \geq 0, \bigcap_{j=-k}^l H^{-j}( \Int_\phi  R_{x_j}) \neq \emptyset \right \}. \]
\end{defn}

In \S2.3 of \cite{pollicott}, the symbolic space $\Sigma(\RRR)$ is shown to be a shift of finite type. There is a canonically defined map $\pi: \Sigma(\RRR) \to \bigcup_i R_i$ given by $\pi(\underline x) = \bigcap_{j=-\infty}^\infty H^{-j}(R_{x_j})$. Let $\rho= r \circ \pi: \Sigma \to (0, \infty)$ and let $\Sigma^\rho=\Sigma^\rho(\RRR)$ be the suspension flow over $\Sigma$ with roof function $\rho$. We extend $\pi$ to $\Sigma^\rho$ by $\hat \pi(\underline x, t) = \phi_t(\pi(\underline x))$. Pollicott shows the following.

\begin{thm} \emph{(}\cite[Theorem 1]{pollicott}\emph{)}
If $(\phi_t)$ is a metric Anosov flow on $Y$, and $(\RRR, \SSS)$ is a Markov proper family, then $\Sigma(\RRR)$ is a shift of finite type and the map $\hat\pi: \Sigma^\rho \to Y$ is finite-to-one, continuous, surjective, injective on a residual set,  and satisfies  $\hat \pi \circ f_t = \phi_t \circ \hat \pi$, where $(f_t)$ is the suspension flow.
\end{thm}

We say that a flow has a \emph{strong Markov coding} if the conclusions of the previous theorem are true with the additional hypothesis that the roof function $\rho$ is H\"older and that the map $\hat \pi$ is H\"older. This is condition (III) on p.195 of \cite{pollicott}. Since $\rho= r \circ \pi: \Sigma \to (0, \infty)$, it suffices to know that $\hat\pi$ is H\"older and $r$ is H\"older where it is continuous. Thus, we can formulate Pollicott's result as follows:

\begin{thm} [Pollicott] \label{t.pollicott}
If $(\phi_t)$ is a metric Anosov flow, and there exists a Markov proper family $(\RRR, \DDD)$ such that the return time function $r$ for $\RRR$ is H\"older where it is continuous, and the natural projection map $\hat \pi: \Sigma^\rho \to X$ is H\"older, then the flow has a strong Markov coding. 
\end{thm}

A drawback of this statement is that it is not clear how to meet the H\"older requirement of these hypotheses. Our Theorem \ref{t.usefulversion} is designed to remedy this. Recall the hypotheses of Theorem \ref{t.usefulversion} are that the metric Anosov flow is H\"older and that there exists a pre-Markov proper family $(\BBB, \DDD)$ so that the return time function and the projection maps to the $B_i$ are H\"older. We now prove Theorem \ref{t.usefulversion} by showing that these hypotheses imply the hypotheses of Theorem \ref{t.pollicott}.

\begin{proof}[Proof of Theorem \ref{t.usefulversion}.]

We verify the hypotheses of Theorem \ref{t.pollicott}. Let the family $(\RRR, \SSS)$ be the Markov family provided by applying Lemma \ref{PtoB} to $(\BBB, \DDD)$. Recall that by Lemma \ref{PtoB}, we can choose the scale $\alpha$ for $(\BBB, \DDD)$ as small as we like. Then $\RRR$ consists of rectangles $R_i$ which are subsets of elements of $\BBB$ shifted by the flow for some small time. Thus, the return time function for $\RRR$ inherits H\"older regularity from the return time function for $\BBB$.

Now we use Theorem \ref{expclose}  to show that the projection map $\pi$ from $\Sigma(\RRR)$ is H\"older. Fix some small $\alpha_0>0$. Choose $\epsilon>0$ sufficiently small that the projection maps to any section $S$ with diameter less than $\alpha_0$ are well-defined on $\phi_{[-\epsilon,\epsilon]}S$. Then let us suppose that our Markov family is at scale $\alpha$ so small that $\alpha<\alpha_0$ and $3\alpha < \delta$ where $\delta$ is given by Theorem \ref{expclose} for the choice of $\epsilon$ above. Let $\underline i, \underline j \in \Sigma(\RRR)$ which agree from $i_{-n}$ to $i_n$. We write $x, y$ for the projected points, which belong to some $B_{i^*}$. If two orbits pass through an identical finite sequence $R_{i_{-n}}, \ldots, R_{i_0}, \ldots R_{i_n}$ then they are  $3\alpha$-close for time at least $2n$ multiplied by the minimum value of the return map on $\RRR$, which we write $r_0$. The distance is at most $3 \alpha$ since $\diam R_i < \alpha$ and the return time is less than $\alpha$. Thus, by Theorem \ref{expclose} there is a time $v$ with $|v|<\epsilon$ so that $d(x, \phi_v y)<  \alpha e^{-\lambda 2nr_0}$. Using H\"older continuity of the projection map $\Proj_{R_{i_0}}$, which is well-defined at $\phi_vy$ since $|v|<\epsilon$, we have 
\[
d(x,y) = d(\Proj_{R_{i_0}}x, \Proj_{R_{i_0}}\phi_v y)< Cd(x, \phi_vy)^\beta,
\] 
where $\beta$ is the H\"older exponent for the projection map. Thus, $d(x,y) < C\alpha e^{-(2\beta \lambda r_0) n}$. Since $d(\underline i, \underline j) = 2^{-n}$, this shows the projection $\pi$ from $\Sigma(\RRR)$ is H\"older. 

It follows that the roof function $\rho = \pi \circ r$ is H\"older.  Thus, since $\pi$ is H\"older, the roof is H\"older and the flow is H\"older, it follows that $\hat \pi: \Sigma^\rho \to X$ is H\"older.
\end{proof}

The advantage of the formulation of Theorem \ref{t.usefulversion} is that the hypotheses for the strong Markov coding are now written entirely in terms of properties of the flow and families of sections $\DDD$. In the terminology introduced above, Bowen showed that transitive Axiom A flows admit a strong Markov coding, using smoothness of the flow and taking the sections to be smooth discs to obtain the regularity of the projection and return maps. For a H\"older continuous metric Anosov flow, we do not know of a general argument to obtain this regularity. Our strategy to verify the hypotheses of Theorem \ref{t.usefulversion} in the case of geodesic flow on a $\CAT(-1)$ space is to construct proper families in which the sections are defined geometrically. For these special sections, we can establish the regularity that we need. Our argument relies heavily on geometric arguments which are available for CAT(-1) geodesic flow, but do not apply to general metric Anosov flows.

%

\section{Geometric rectangles and H\"older properties}\label{sec:good rectangles}

\subsection{Geometric rectangles} \label{subsec:geom rect}

In this section, we define geometric rectangles which can be built in $G\tilde X$ for any $\CAT(-1)$ space $\tilde X$.

\begin{defn}\label{def:geom rectangle}
Let $U^+$ and $U^-$ be disjoint open sets in $\partial_\infty \tilde X$. Let $T \subset \tilde X$ be a transversal on $\tilde X$ to the geodesics between $U^-$ and $U^+$ -- that is, a set $T$ so any geodesic $c$ with $c(\infty)\in U^+$ and $c(-\infty) \in U^-$ intersects $T$ exactly once. Let $R(T, U^+, U^-)$ be the set of all geodesics $c$  with $c(\infty)\in U^+$ and $c(-\infty) \in U^-$ and which are parametrized so that $c(0)\in T$. If $R(T,U^+,U^-)$ is a section to the geodesic flow on $G \tilde X$, we call $R(T,U^+,U^-)$ a \emph{geometric rectangle for $G\tilde X$}. Any sufficiently small geometric rectangle in $G\tilde X$ projects bijectively to $GX$, and this defines a \emph{geometric rectangle for $GX$}.
\end{defn}

\begin{figure}[h]
\begin{tikzpicture}

\draw (0,0) circle (3.5);

\node at (0,4) {$U^+$};
\node at (0,-4) {$U^-$};

\draw (-1,-3.35) arc (-19.7:19.7:10cm);
\draw (1,3.35) arc (160.3:199.7:10cm);

\draw[ultra thick, red] (1.75,3.05) arc (60:120:3.5cm);
\draw[ultra thick, red] (-1.75,-3.05) arc (240:300:3.5cm);

\draw [thick, ->] (-.43,0) -- (-.43,.3) ;
\draw [thick, ->] (.43,0) -- (.43,.3) ;
\draw [thick, ->] (0,0) -- (-.1,.3) ;

\draw[ultra thick] (-.5,0) -- (.5,0);

\node at (.8,0){$T$};

\node at (-1,2.7){$c'$};
\node at (1,2.7){$c$};

\draw (-1,3.35) --(1,-3.35);

\node at (.2,-1.3){$d$};

\end{tikzpicture}
\caption{Illustrating Definition \ref{def:geom rectangle}. The arrows mark the basepoint and direction for each geodesic in $R(T,U^+,U^-)$.}\label{fig:geom rectangle}
\end{figure}
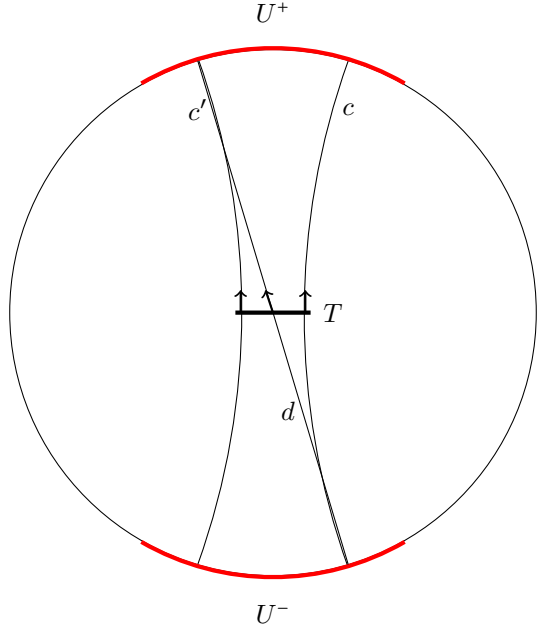
If $c, c' \in R(T,U^+,U^-)$, then $\Proj_T \langle c, c'\rangle $ is the geodesic $d$ which connects the backward endpoint of $c$ to the forwards endpoint of $c'$, with $d(0)\in T$, and thus $R(T,U^+,U^-)$ is a rectangle in the sense of Definition \ref{def:rectangle} (see Figure \ref{fig:geom rectangle}). In the case when $X$ is convex cocompact, we observe that $R(T, U^+, U^-)\cap \hat G\tilde X$ is still a rectangle. This is because membership of $\hat G\tilde X$ is determined by whether the endpoints of a geodesic lie in $\Lambda \subset \partial^\infty\tilde X$, and thus if $c, c'\in \hat G\tilde X$, then $\langle c,c'\rangle \in \hat G\tilde X$. We keep the notation $R(T, U^+, U^-)$ for rectangles in $\hat G \tilde X$. Although this is formally a slight abuse of notation, using the same notation for rectangles in $G \tilde X$ and $\hat G \tilde X$ simplifies notation and will not cause any issues.

To build rectangles we need to specify the sets $U^+$ and $U^-$ and choose our transversals. We do so in the following definition.

Fix a parameter $\tau >>1$. Let $c\in G\tilde X$. Let $B_1 = B_{d_{\tilde X}}(c(-\tau),1)$ and $B_2 = B_{d_{\tilde X}}(c(\tau),1)$ be the open balls of $d_{\tilde X}$-radius 1 around $c(\pm\tau)$. Let 
\[ \gamma(c, \tau) = \{ c'\in GX : c' \cap B_i \neq \emptyset \mbox{ for } i=1,2 \}. \]
Let 
\[  \partial (c, \tau) = \{ (c'(-\infty), c'(+\infty)) \in \partial^\infty \tilde X  \times  \partial^\infty \tilde X  : c'\in \gamma(c,\tau)  \}.\]
It is easy to check that $\partial (c, \tau)$ is open in the product topology on $\partial^\infty \tilde X \times \partial^\infty \tilde X$. Then we may find open sets $U^-$ and $U^+$ such that $(c(-\infty), c(+\infty)) \in U^- \times U^+ \subset \partial (c, \tau)$.

\begin{defn}\label{def:good rectangle}
Let $c\in G\tilde X$ and $\tau >>1$. Let $U^-$ and $U^+$ satisfy $U^- \times U^+ \subset \partial (c, \tau)$. The \emph{good rectangle}  $R(c, \tau; U^-, U^+)$ is the set of all $\eta \in G \tilde X$ which satisfy:
\begin{enumerate}
	\item $\eta(-\infty) \in U^-$ and $\eta(+\infty) \in U^+$,
	\item $B_{c}(\eta(0)) = 0$,
	\item If $\eta(t_1)\in B_1$ and $\eta(t_2)\in B_2$, then $t_1<0<t_2$.
\end{enumerate}
In the convex cocompact case, in addition we take the intersection of all such geodesics with $\hat G\tilde X$. To remove arbitrariness in the choice of $U^-, U^+$, we can let $\delta>0$ be the biggest value so that if $U^-_\delta = B_\infty(c(+\infty), \delta)$ and $U^+_\delta = B_\infty(c(-\infty), \delta)$, then $U^-_\delta \times U^+_\delta \subset \partial(c, \tau)$. We can set $R(c, \tau)= R(c, \tau; U^-_\delta, U^+_\delta)$. 
\end{defn}

In other words, for \emph{good} rectangles, we take as our transversal $T$ on $X$ a suitably sized disc in the horocycle  based at $c(+\infty)$ through $c(0)$ (see Figure \ref{fig:good rectangle}).

We will usually consider the `maximal' good rectangle $R(c, \tau)$. However, we note that the definition makes sense for any $V^- \times V^+ \subset \partial (c, \tau)$. In particular, it is not required that the geodesic $c$ itself (which defines the horocycle that specifies the parameterization of the geodesics) be contained in $R(c, \tau; V^-, V^+)$.

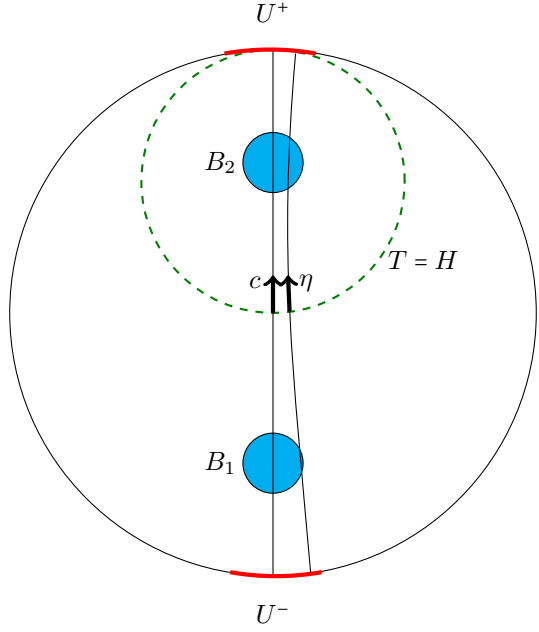
\begin{figure}[h]
\begin{tikzpicture}

\draw (0,0) circle (3.5);

\draw [fill=cyan] (0,2) circle (.4);
\draw [fill=cyan] (0,-2) circle (.4);

\draw[thick, dashed, green!50!black] (0,1.75) circle(1.75);

\node at (0,4) {$U^+$};
\node at (0,-4) {$U^-$};

\draw (0,-3.5) -- (0,3.5);

\draw[ultra thick, red] (.57,3.45) arc (80:100:3.5cm);
\draw[ultra thick, red] (-.57,-3.45) arc (260:280:3.5cm);

\draw [ultra thick,->] (0,0) -- (0,.5);

\draw  (.5,-3.45) to [out=95,in=265] (.3,3.450);

\draw [ultra thick,->] (0.22,0.02) -- (0.2,.5);

\node at (-.24,.4){$c$};
\node at (.44,.4){$\eta$};
\node at (2,.7){$T=H$};
\node at (-.7,2){$B_2$};
\node at (-.7,-2){$B_1$};

\end{tikzpicture}
\caption{A geodesic $\eta\in R(c, \tau ; U^-,U^+)$ as in Definition \ref{def:good rectangle}.}\label{fig:good rectangle}
\end{figure}

To justify this definition, we must verify that $R(c, \tau; U^-, U^+)$ is in fact a rectangle in the sense of Definition \ref{def:geom rectangle}. That is, we need to prove the following two lemmas:

\begin{lem} \label{sectionproof}
For any $\eta \in G\tilde X$ with $\eta(-\infty) \in U^-$ and $\eta(+\infty) \in U^+$, there is exactly one point $p \in \eta$ such that $B_{c}(p)=0$ and such that $p$ lies between $\eta$'s intersections with $B_1$ and $B_2$.
\end{lem}

\begin{proof}
We have $B_{c(0)}(\eta(t_1),c(+\infty))>0$ when $\eta(t_1) \in B_1$ and $B_{c(0)}(\eta(t_2),c(+\infty))<0$ when $\eta(t_2) \in B_2$. Continuity and convexity of the Busemann function implies that there is a unique $t^*\in (t_1, t_2)$ such that $B_{c(0)}(\eta(t^*),c(+\infty))=0$. Let $p=\eta(t^*)$.  
\end{proof}
\begin{lem}
$R(c, \tau; U^-, U^+)$  is a section.
\end{lem}
\begin{proof}
The openness of $U^-$ and $U^+$, and the $1$-Lipschitz property of Busemann functions are the key facts.
\end{proof}

We give the following distance estimates for geodesics in a rectangle.

\begin{lem} \label{l.distance}
For all $\eta \in R(c, \tau; U^-, U^+)$, we have
\begin{enumerate}
	\item $d_{\tilde X}(c(0), \eta(0))\leq 2$;
	\item $d_{\tilde X}(c( \pm \tau), \eta(\pm \tau))<4$.
\end{enumerate}
\end{lem}

\begin{proof}

First, we prove (1). By the definition of the rectangle, we know that there exist times $t^+>0$ and $t^-<0$ so that $d_{\tilde X}(c(\tau), \eta(t^+))<1$ and $d_{\tilde X}(c(-\tau), \eta(t^-))<1$. Since the distance between two geodesic segments is maximized at one of the endpoints, we know that $d_{\tilde X}(\eta(0), c)<1$. Thus, there exists $t^\ast$ so that $d_{\tilde X}(\eta(0), c(t^\ast))<1$. Thus, $d_{\tilde X}(c(0), \eta(0)) \leq d_{\tilde X}(c(0), c(t^\ast))+ d_{\tilde X}(\eta(0), c(t^\ast)) < |t^\ast|+1$.

Since the Busemann function is $1$-Lipschitz,
\[|B_c(c(t^\ast))| = |B_c(c(t^\ast))- B_c(\eta(0))| \leq d_{\tilde X}(c(t^\ast), \eta(0)) <1.\]
Since $|B_c(c(t^\ast))| = |t^\ast|$, it follows that $|t^\ast|<1$. Thus, $d_{\tilde X}(c(0), \eta(0))<2$.

We use (1) to prove (2). Observe that $t^+\leq \tau+3$. This is because
\begin{align*}
	t^+=d_{\tilde X}(\eta(0), \eta(t^+)) &\leq d_{\tilde X}(\eta(0), c(0))+ d_{\tilde X}(c(0), c(\tau)) + d_{\tilde X}(c(\tau), \eta(t^+)) \\ 
			& \leq 2+\tau +1.
\end{align*}
We also see that $t^+\geq \tau-3$. This is because
\begin{align*}
	\tau = d_{\tilde X}(c(0), c(\tau)) & \leq d_{\tilde X}(c(0), \eta(0))+d_{\tilde X}(\eta(0), \eta(t^+))+d_{\tilde X}(\eta(t^+), c(\tau)) \\
			& \leq 2 + t^+ +1.
\end{align*}
Thus $|\tau-t^+| <3$. It follows that
\[d_{\tilde X}(c(\tau), \eta(\tau)) \leq d_{\tilde X}(c (\tau), \eta(t^+))+ d_{\tilde X}(\eta(\tau), \eta(t^+))< 1+ 3 = 4.\]
The argument that $d_{\tilde X}(c(-\tau), \eta(-\tau))<4$ is analogous.
\end{proof}

We obtain linear bounds on the Busemann function for $\eta \in R(c, \tau; U^-,U^+)$.

\begin{lem}\label{lem:linear}
For all $\eta \in R(c, \tau; U^-, U^+)$,
\[   -t \leq B_{c(0)}(\eta(t),c(+\infty)) \leq  - \frac{t}{2} \mbox{ for all } 0\leq t<\tau   \]
and
\[   -\frac{t}{2} \leq B_{c(0)}(\eta(t),c(+\infty)) \leq  - t \mbox{ for all } -\tau <  t \leq 0.   \]
That is, for times between $-\tau$ and $\tau$, the values of the Busemann function along $\eta$ lie between $-t$ and $-\frac{t}{2}$.
\end{lem}

\begin{proof}

That $-t \leq B_{c}(\eta(t))$ follows immediately from the 1-Lipschitz property of Busemann functions. By Lemma \ref{l.distance} and the 1-Lipschitz property of Busemann functions, $|B_{c}(\eta(\tau)) + \tau | = | B_{c}(\eta(\tau)) - B_{c}(c(\tau)) |  < 4$, and similarly $|B_{c}(\eta(\tau)) - \tau | = | B_{c}(\eta(-\tau)) - B_{c}(c(-\tau)) |  < 4$. Therefore, $f(t) = B_{c}(\eta(t))$ is a convex function with $f(-\tau)\in (\tau-4,\tau]$, $f(0)=0$ and $f(\tau) \in [-\tau,-(\tau-4))$. Then if for some $t_0 \in (-\tau,0]$, $f(t_0) < -\frac{t_0}{2}$, or for some $t_0 \in [0,\tau)$, $f(t_0)>-\frac{t_0}{2}$, then for all $t>\max\{0,t_0\}$, by convexity, $f(t) > -t/2$. But then $f(\tau) > -\frac{\tau}{2}$, a contradiction since $\tau>>1$. 
\end{proof}

The proof actually yields the upper bound of $B_{c}(\eta(t)) \leq  - \frac{\tau-4}{\tau} t$ but all we need is some linear bound with non-zero slope.

\begin{lem}\label{lem:deform tool}
Let $R_1$ and $R_2$ be rectangular subsets of good geometric rectangles. Suppose $\diam(R_1) = \epsilon$ and that $R_1 \cap R_2 \neq \emptyset$. Then for $|t|>2L\epsilon$, $g_tR_1 \cap R_2 = \emptyset$, where $L$ is the constant from Lemma \ref{lem:X GX bound}.
\end{lem}

\begin{proof}
Let $f$ be the Busemann function used to specify the basepoints of geodesics in $R_2$. Since the diameter of $R_1$ is $\epsilon$, and since for some $\eta\in R_1$, $f(\eta(0))=0$, $|f(c(0))| < L\epsilon$ for all $c\in R_1$. This uses Lemma \ref{lem:X GX bound} and the 1-Lipschitz property of Busemann functions with respect to $d_X$. If $\eta \in R_1 \cap R_2$, then $\frac{1}{2}|t| \leq |f(\eta(t))| \leq |t|$ by Lemma \ref{lem:linear}. Now suppose that $\eta \in g_tR_1 \cap R_2$ for some $|t|>2L\epsilon$. Then $g_{-t}\eta \in R_1$ and we must have $|f(\eta(-t))| > L\epsilon$, which is a contradiction.
\end{proof}

%

\subsection{H\"older properties} 

We are now ready to prove the regularity results we need to apply Theorem \ref{t.usefulversion}. First, we show that return times between geometric rectangles are Lipchitz. Let $R=R(c, \tau; U^+, U^-)$ and $R'=R(c', \tau'; U'^+, U'^-)$ be good geometric rectangles and let $d\in R$ such that $g_{t_0}d \in R'$ for some $t_0$ which is minimal with respect to this property. We write $r(d)=r(d, R, R') :=t_0$; this is the return time for $d$ to $R \cup R'$.

Let us make the standing assumption that all return times are bounded above by $\alpha>0$. Note that  $d\in R$ and $g_{t_0}d \in R'$ iff $d(-\infty)\in U^-\cap U'^-$ and $d(+\infty)\in U^+\cap U'^+$. The key property we want is the following: 

\begin{prop}\label{prop:Lipschitz}
Let $R, R'$ be good rectangles and $Y= R \cap H^{-1}(R')$. Then the return time map $r:(Y, d_{G\tilde X}) \to \mathbb{R}$ is Lipschitz.
\end{prop}

\begin{proof}

Let $v,w \in Y$ with return times $r(v), r(w)$, respectively. Let $\epsilon = d_{GX}(v,w)$. We consider the Busemann function determined by the geodesic $c'$ which defines the rectangle $R'$.

Let $f(t) = B_{c'}(v(r(v)+t))$ and let $g(t) = B_{c'}(w(r(v)+t))$. Then $r(w)-r(v) = t^*$ where $t^*$ is the unique value of $t$ with $|t^\ast|< \alpha$ such that $g(t^\ast) = 0$.  By Lemma \ref{lem:linear}, the graph of $f(t)$ lies between the lines $y=-t$ and $y=-\frac{t}{2}$ for small $t$.

 Let $C = e^\alpha$, where $\alpha$ is an upper bound on the return time. By Lemmas \ref{lem:X GX bound} and \ref{lem:Lipschitz}, $d_X(v(s), w(s))<LC\epsilon$ for all $s<\alpha$, where $C$ is a uniform constant. The 1-Lipschitz property of Busemann functions implies that $|f(t)-g(t)| < LC\epsilon$. 

 Thus, for $t>0$, we have $g(t)\leq f(t)+LC \epsilon \leq -t/2 +LC\epsilon$, and so for $t>2LC \epsilon$, $g(t)<0$. For $t<0$, we have $g(t) \geq f(t)-LC \epsilon \geq -t/2-LC \epsilon$, and so for $t<-2LC \epsilon$, we have $g(t)>0$. Thus, by the intermediate value theorem, the root $g(t^\ast)=0$ satisfies $t^\ast \in (-2LC\epsilon, 2LC \epsilon)$. Therefore, $|r(w)-r(v)| = |t^\ast| < 2LC\epsilon$ proving the desired Lipschitz property with constant $2LC$. 
\end{proof}

We now show that the projection map to a good rectangle is H\"older. Consider any good geometric rectangle $R=R(c,\tau; U^-, U^+)$. Fix some small $\alpha>0$ so that $(-\alpha, \alpha) \times R \to G\tilde X$ by $(t,x) \mapsto g_t x$ is a homeomorphism.

 \begin{prop} \label{projholder}
 $\Proj_R:g_{(-\alpha, \alpha)}R \to R$ is H\"older.
 \end{prop}

\begin{proof}
We prove that for all $x,y\in g_{(-\alpha, \alpha)}R$ there exists some $K>0$ such that 
\[ d_{G\tilde X} (\Proj_Rx, \Proj_Ry) \leq K d_{G\tilde X}(x,y)^{\frac{1}{2}}.\]

First, note that for all $|t|<2\alpha$, $g_t$ is a $e^{2\alpha}$-Lipschitz map by Lemma \ref{lem:Lipschitz}. Therefore, to prove the Proposition, it suffices to prove the case where $x\in R$, as we can pre-compose the projection in this case with the Lipschitz map $g_{t^*}$ where $g_{t^*}x\in R$.

Let $t=B_c(y(0))$. By Lemma \ref{lem:linear} for all $|s|<\tau$, $\frac{|s|}{2} \leq |B_c(y(s))-t| \leq |s|.$ Similarly, $\frac{|s|}{2} \leq |B_c(x(s))| \leq |s|$. Since $B_c(x(s))$ and $B_c(y(s))$ are both decreasing by definition of $R$, these inequalities give us that
\[ |B_c(x(s)) - B_c(y(s))| \geq t-\frac{|s|}{2} \ \mbox{ for all } s\leq \tau.\] 
Since $B_c$ is a 1-Lipschitz function on $\tilde X$, 
\[ d_{\tilde X}(x(s), y(s)) \geq  t-\frac{|s|}{2} \ \mbox{ for all } s\leq \tau. \]
Then we can compute
\[ d_{G\tilde X}(x,y) \geq \int_{-2t}^{2t} \left(t-\frac{|s|}{2}\right) e^{-2|s|} ds = \frac{1}{4}(-1+e^{-4t}+4t) \geq c t^2\] 
for a properly chosen $c>0$ since $|t|<\alpha$. By Lemma \ref{lem:linear} and the fact that the geodesic flow moves at speed one for $d_{G\tilde X}$, $d_{G\tilde X}(y,\Proj_Ry) \leq 2|t|$. Using  $d_{G\tilde X}(x, \Proj_Ry)  \leq d_{G\tilde X}(x,y) + d_{G\tilde X}(y,\Proj_Ry)$, if there exists some $L>0$ such that $d_{G\tilde X}(y,\Proj_Ry) \leq Ld_{G\tilde X}(x,y)^{1/2}$, the Lemma is proved. But we have shown above that $d_{G\tilde X}(x,y) \geq ct^2$ and $d_{G\tilde X} (y, \Proj_Ry)\leq 2|t|$.
\end{proof}

%

\subsection{A pre-Markov proper family of good rectangles}\label{sec:good markov}

To complete our argument, it suffices to check that a pre-Markov proper family $(\RRR, \SSS)$ can be found where the family of sections $\SSS$ consists of good geometric rectangles, perhaps flowed by a small time. Applying the results of the previous section, this will show that $(\RRR, \SSS)$ has properties (1) and  (2) of Theorem \ref{t.usefulversion}.

\begin{prop}\label{prop:good mpf}
Let $X$ be a convex cocompact locally $\CAT(-\kappa)$ space, and let $(g_t)$ be its geodesic flow (on $GX$ when $X$ is compact, and on $\hat GX$ otherwise). For any sufficiently small $\alpha>0$, there exists a pre-Markov proper family  $(\BBB, \DDD)$ for the flow at scale $\alpha$ such that each $D_i$ has the form $g_{s_i}R_i$ for some $s_i$ with $|s_i| << \alpha$ and some good geometric rectangle $R_i$.
\end{prop}

We need the following lemma.

\begin{lem} \label{markovpf}
Let $(\phi_t)$ be a Lipschitz continuous expansive flow on a compact metric space. Given a proper family $(\BBB, \DDD)$ for $(\phi_t)$ at scale $\alpha>0$ where the $B_i$ and $D_i$ are rectangles, there exists a pre-Markov proper family $(\BBB', \DDD')$ at scale $\alpha>0$ such that every $D_k'\in\mathcal{D}'$ is the image under $\phi_{s_k}$ of some $D_i\in\mathcal{D}$ where $|s_k| << \alpha$.
\end{lem}

It is clear from the proof below that the times $s_k$ can be made arbitrarily small in absolute value.

\begin{proof}
Let $(\mathcal{B},\mathcal{D})= \{(B_i, D_i) : i = 1, \ldots, n \}$ be a proper family at scale $\alpha$ where the $B_i$ and $D_i$ are rectangles. Recall that by definition, a proper family satisfies:  (1) $\diam (D_i)< \alpha$ and $B_i \subset D_i$ for each $i \in \{1, 2, \ldots, n\}$; (2) $\bigcup_{i=1}^n\phi_{(-\alpha, 0)} (\Int_\phi B_i)=Y$; (3) for all $i\neq j$, if $\phi_{[0, 4\alpha]}(D_i) \cap D_j \neq \emptyset$, then $\phi_{[-4\alpha, 0]}(D_i) \cap D_j = \emptyset$. 

Our strategy for constructing new proper families out of $(\BBB, \DDD)$ is to replace an element $(B_i, D_i)$ by a finite collection $\{ (\phi_{s_k}R_k, \phi_{s_k}D_i) \}_k$ where $R_k$ are rectangles with $R_k \subset B_i$ and  $\bigcup_k \Int_\phi R_k$ covers $\Int_\phi B_i$. Then $\phi_{s_k}R_k$ and $\phi_{s_k}D_i$ inherit the rectangle property from $R_k$ and $D_i$ (and are closed if $R_k$ and $D_i$ are), and by choosing all $s_k$ distinct and sufficiently small in absolute value, we can ensure that the resulting collection will still satisfy (1), (2), and (3). We give some details. 

For (1), since the flow is Lipschitz and $\diam(D_{i})<\alpha$, we can choose $\epsilon_1$ so small that $\diam(\phi_{\pm\epsilon_1}D_{i})<\alpha$. Thus, (1) will be satisfied if all $s_k$ have $|s_k| < \epsilon_1$.

For (2), since $\bigcup_k \Int_\phi R_k$ covers $\Int_\phi B_i$, then it suffices to assume that all $s_k$ are sufficiently small in absolute value.

For (3), let $\beta>0$ be the minimum value of $s$ so that there is a pair $D_j, D_k$ in our proper family with both $\phi_{[0,s]} D_j \cap D_k$ and $\phi_{[-s,0]} D_j \cap D_k$ nonempty, and observe that we must have $\beta>4\alpha$. Choosing $\epsilon_2$ smaller than $\frac{\beta-4\alpha}{2}$ and smaller than $d(D_j, D_k)$ for any $j\neq k$, condition (3) will be satisfied for $\phi_{s_j} D_j$ and $\phi_{s_k}D_k$ when $|s_j|, |s_k| < \epsilon_2$ and $j\neq k$. Also, it is clear that (3) will hold for the pair $\phi_{s_{k_1}} D$ and $\phi_{s_{k_2}}D$ when $D \in \DDD$, $|s_{k_1}|, |s_{k_2}| < \epsilon_2$ and $s_{k_1} \neq s_{k_2}$.

We now use this strategy to refine $(\BBB, \DDD)$ to ensure the pre-Markov property \eqref{premarkov} holds. For $B \in \BBB$, consider the set
\[ F(B; \mathcal{B},\mathcal{D}) = \{B_j \in \BBB : \ B \cap \phi_{[-2\alpha, 2\alpha]}B_j \neq \emptyset \mbox{ but } B \nsubseteq \phi_{(-3\alpha,3\alpha)} D_j\}.\]

The set $F(B; \mathcal{B},\mathcal{D})$ is finite and encodes the elements of the proper family for which an intersection with $B$ causes an \emph{open} version of \eqref{premarkov} to fail. Clearly if $F(B; \mathcal{B},\mathcal{D}) = \emptyset$ for all $B \in \BBB$, then the pre-Markov condition \eqref{premarkov} is satisfied. 

Let $i_1< i_2 < \cdots <i_n$ be the set of all indices so that $F(B_{i_j}; \mathcal{B},\mathcal{D}) \neq \emptyset$. 
We cover $B_{i_1}$ by a finite collection of rectangles $R_k \subset B_{i_1}$ such that
\begin{itemize}
	\item $\bigcup_k \Int_\phi R_k$ covers $\Int_\phi  B_{i_1}$, and
	\item If $R_{k} \cap \phi_{[-2\alpha, 2\alpha]}B_j \neq \emptyset$ for some $j$, then $R_{k}\subseteq \phi_{(-3\alpha,3\alpha)} D_j$.
\end{itemize}

\begin{figure}
\begin{center}
\begin{tikzpicture}[scale=.9]

\draw[thick, orange!20!white, fill=orange!20!white] (-2.8,1.42) -- (-.12,1.42) -- (-.33,1) -- (-3,1) ;
\draw[thick, orange!20!white, fill=orange!20!white] (-1.6,1.1) -- (-.75,2.8) -- (.5,2.8) -- (-.38,1.1) -- (-1.6,1.1) ;
\draw[thick, orange!20!white, fill=orange!20!white] (2.3,1.6) -- (-.6,1.6) -- (-.1,2.65) -- (2.84,2.65) ;
\draw[thick, orange!20!white, fill=orange!20!white] (-.52,1) -- (-.09,1.8) -- (2.4,1.8) -- (2,1);
\draw[thick, orange!20!white, fill=orange!20!white] (-2.8,1.35) -- (-1.4,1.35) -- (-.75,2.65) -- (-2.2,2.65) ;
\draw[thick, orange!20!white, fill=orange!20!white] (-2.22,2.57) -- (2.77,2.57) -- (3,3) -- (-2,3) ;
\draw[thick, orange!20!white, fill=orange!40!white] (-2.2,2.65) -- (-.75,2.65) -- (-1.4,1.35) -- (-1.47,1.35) -- (-.87,2.57) -- (-2.22,2.57) ;
\draw[fill=orange!40!white] (-2.8,1.42) -- (-.12,1.42) -- (-.33,1) -- (-.52,1) -- (-.47,1.1) -- (-1.6,1.1) -- (-1.47,1.35) -- (-2.8,1.35) ;
\draw[fill=orange!40!white] (2.3,1.6) -- (-.6,1.6) -- (-.14,2.57) -- (-.86,2.57) -- (-.75,2.8) -- (.5,2.8) -- (.43,2.65) -- (2.84,2.65) -- (2.77,2.57) -- (0.38,2.57) -- (-0.02,1.8) -- (2.4,1.8);
\draw[fill=orange!40!white] (-.52,1) -- (-.09,1.8) -- (-.03,1.8) -- (-.38,1.09);

\draw[thick, fill=blue!50!white] (-2.75,1.5) -- (-1.5, 1.5) -- (-1,2.5) -- (-2.25, 2.5) ;

\draw  [thick, dashed] plot [smooth cycle] coordinates { (-2,4) (-3.3,3) (-4,1) (-3,0) (2,0) (3.3,1) (4,3) (3,4) } ;

\draw [thick, fill=pink] (0,-2) -- (-.5,-3) -- (2.5,-3) -- (3,-2) -- (0,-2) ;

\draw  [thick, red, dashed] plot [smooth cycle] coordinates { (0,-1.5) (-.65,-2) (-1,-3) (-.5,-3.5) (2.5,-3.5) (3.2,-3) (3.5,-2) (3,-1.5) } ;

\draw [thick, fill=pink] (-.25, 1) -- (0,1.5) -- (2.25, 1.5) -- (2,1) -- (-.25,1) ;

\draw  [thick, blue, dashed] (-2.9,1.2) .. controls (0,.8) and (.3,3.4) .. (-2.1,2.8) ;
\draw  [thick, red, dashed] (-.75,1) .. controls (-.4,2.1) and (0,2) .. (2.5,2) ;

\draw  (0,-2) -- (0,-.12) ;
\draw [->] (0,-2) -- (0,-.5) ;
\draw [dashed] (0,-.12) -- (0,1.5) ;

\draw [thick, fill=blue!50!white, xshift=-4cm, yshift=8cm] (0,-2) -- (-.5,-3) -- (2.5,-3) -- (3,-2) -- (0,-2) ;
\draw  [thick, blue, dashed, xshift=-4cm, yshift=8cm] plot [smooth cycle] coordinates { (-0.1,-1.5) (-.75,-2) (-1.1,-3) (-.6,-3.5) (2.5,-3.5) (3.2,-3) (3.5,-2) (3,-1.5) } ;

\draw[thick, orange] (-.52,1) -- (-.09,1.8) -- (2.4,1.8);
\draw[thick, orange] (-2.8,1.35) -- (-1.4,1.35) -- (-.75,2.65) -- (-2.2,2.65) ;
\draw[thick, orange] (-2.22,2.57) -- (2.77,2.57) ;
\draw[thick, orange] (-2.8,1.42) -- (-.12,1.42) -- (-.33,1) ;
\draw[thick, orange] (-1.6,1.1) -- (-.75,2.8) -- (.5,2.8) -- (-.38,1.1) -- (-1.6,1.1) ;
\draw[thick, orange] (2.3,1.6) -- (-.6,1.6) -- (-.1,2.65) -- (2.84,2.65) ;

\draw [thick] (-2,3) -- (-3,1) -- (2,1) -- (3,3) -- (-2,3) ;

\draw[->] (-1.5,1.5) -- (-1.5, 3.5) ;
\draw (-1.5,1.5) -- (-1.5, 4.5) ;
\draw[dashed] (-1.5,4.5) -- (-1.5, 5) ;

\draw[->] (-1,2.5) -- (-1, 4.5) ;
\draw (-1,2.5) -- (-1, 4.73) ;
\draw[dashed] (-1,4.7) -- (-1, 6) ;

\node at (3,2) {$B_1$} ;
\node at (3.8,1) {$D_1$} ;

\node at (3,-2.7) {$B_j$} ;
\node at (3.9,-1.7) {$D_j$} ;

\node at (-4.55,5.7) {$B_i$} ;
\node at (-4.9,4.2) {$D_i$} ;

\end{tikzpicture}
\end{center}
\caption{Ensuring the pre-Markov condition \eqref{premarkov}. $(1,j)$ and $(1,i)$ belong to $F(B_1;\mathcal{B},\mathcal{D})$. The flow direction is vertical. The orange rectangles provide one possible choice for the $R_k$.}\label{fig:cut}
\end{figure}
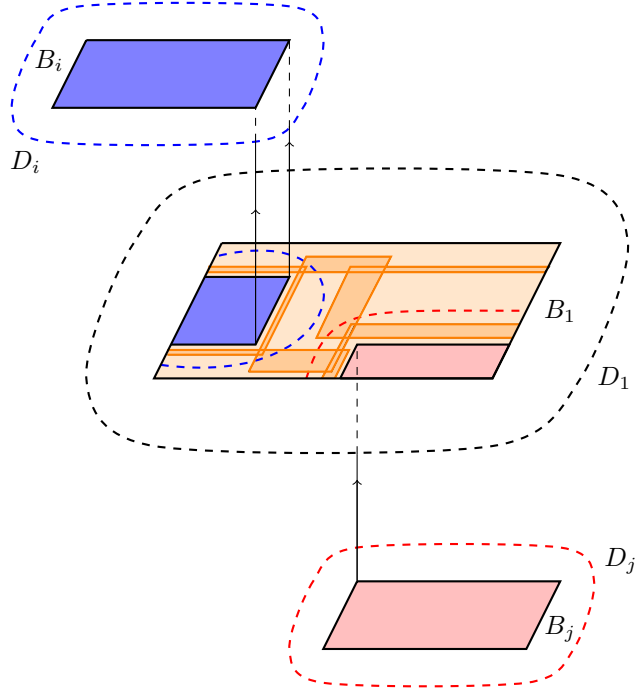

It is clearly possible to find collections of rectangles satisfying the first condition. The second can be satisfied because $B_{i_1} \cap \phi_{[-2\alpha, 2\alpha]}B_j$ is a closed subset of $B_{i_1}$ contained in the open subset $\phi_{(-3\alpha,3\alpha)} D_j$, with respect to the subspace topology on $B_{i_1}$. 
We replace $(B_{i_1}, D_{i_1})$ in $(\mathcal{B}, \mathcal{D})$ with $\{(\phi_{s_{k}}R_{k}, \phi_{s_{k}} D_{i_1})\}_k$ for distinct times  $s_{k}$ sufficiently small in absolute value as detailed above. We obtain $(\mathcal{B}^1, \mathcal{D}^1)$ with $\mathcal{B}^1$ consisting of closed rectangles satisfying conditions (1), (2), and (3). 

To establish condition \eqref{premarkov}, we must prove two things. First, we claim that for all $k$, $F(\phi_{s_{k}}R_{k}; \mathcal{B}^1, \mathcal{D}^1) = \emptyset$. This is true for the following reasons. First, if $B_j \in \BBB$ with $i_1 \neq j$, then by construction we know that $B_j \notin F(\phi_{s_{k}}R_{k}; \mathcal{B}^1,\mathcal{D}^1)$. It remains only to consider sets of the form $\phi_{s_{i}}R_{i}$ with $i \neq k$. Suppose $\phi_{s_{k}}R_{k} \cap \phi_{[-2\alpha, 2 \alpha]} \phi_{s_{i}}R_{i} \neq \emptyset$. Then it is clear that since $|s_i|, |s_k|$ are small and $R_k \subset D_{i_1}$, then  $\phi_{s_{k}}R_{k} \subset \phi_{(-3 \alpha, 3 \alpha)} \phi_{s_i} D_{i_1}$. It follows that $\phi_{s_{i}}R_{i} \notin F(\phi_{s_{k}}R_{k}; \mathcal{B}^1,\mathcal{D}^1)$. We conclude that $F(\phi_{s_{k}}R_{k}; \mathcal{B}^1, \mathcal{D}^1) = \emptyset$.
That is, we have eliminated the `bad' rectangle $B_{i_1}$ from the proper family and replaced it with a finite collection of rectangles that do not have any `bad' intersections.

Second, we claim that for all $k$ and any $j \neq i_1, i_2, \ldots, i_n$, we have that $\phi_{s_k}R_k \notin F(B_j; \mathcal{B}^1,\mathcal{D}^1)$. This is true for the following reason. Since $j \neq i_l$, $F(B_j;\mathcal{B},\mathcal{D}) = \emptyset$. Therefore $B_{i_1}\notin F(B_j; \mathcal{B},\mathcal{D})$ prior to refining and replacing $B_{i_1}$. Therefore, either $B_j \cap \phi_{[-2\alpha,2\alpha]}B_{i_1} = \emptyset$ or $B_j\subseteq \phi_{(-3\alpha,3\alpha)}D_{i_1}$. Either condition is `open,' in the sense that there is some $\epsilon_3(j)>0$ such that the condition remains true if $(B_{i_1},D_{i_1})$ is replaced by $(\phi_s B_{i_1}, \phi_s D_{i_1})$ for $|s|<\epsilon_3(j)$. Therefore, if we further demand that all $s_k$ satisfy $|s_k|<\min_j \epsilon_3(j)$, we will have that $\phi_{s_k}R_k \notin F(B_j; \mathcal{B}^1,\mathcal{D}^1)$, as desired. This implies that $F(B_j;\mathcal{B}^1,\mathcal{D}^1)=\emptyset$ for all such $j$.

From these two facts we conclude that the set of $B\in\mathcal{B}^1$ for which $F(B;\mathcal{B}^1,\mathcal{D}^1)\neq \emptyset$ is (at most) $B_{i_2}, \ldots B_{i_n}$. To  complete the proof, we carry out the `refine-and-replace' scheme finitely many times, modifying $(\mathcal{B}^1, \mathcal{D}^1)$ into $(\mathcal{B}^2, \mathcal{D}^2)$ by carrying out the procedure above on $(B_{i_2}, D_{i_2})$, etc. Finally, we modify $(B_{i_n}, D_{i_n})$ to produce a collection $(\mathcal{B}^n, \mathcal{D}^n)$ which by construction satisfies $F(B; \mathcal{B}^n,\mathcal{D}^n) = \emptyset$ for all $B \in \BBB^n$. In other words, we have eliminated every intersection which causes the pre-Markov property to fail, and this completes the proof. 
\end{proof}

For completeness, we remark on how to obtain the intermediate family of sections $\KKK$ described after Lemma \ref{PtoB}. We choose closed rectangles $K_i$ so $K_i \subset \Int_\phi B_i$. They can be chosen as close to $B_i$ as we like so that $\{\phi_{[-\alpha, 0]} (\Int_\phi K_1), \ldots, \phi_{[-\alpha, 0]} (\Int_\phi K_n)\}$ is an open cover.  Now take a Lebesgue number $12 \delta$ for this open cover. Then for any $x$, $\overline B(x, 6\delta) \subset \phi_{[-\alpha, 0]} (\Int_\phi K_i)$ for some $i$, and thus $\overline B(x, 6\delta) \subset \phi_{[-2 \alpha, 2 \alpha]}K_i$. We now prove Proposition \ref{prop:good mpf} by showing that we can ensure the sections $B_i$ are geometrically defined rectangles.

\begin{proof}[Proof of Proposition \ref{prop:good mpf}]
We show that we can construct a proper family out of good geometric rectangles for $\hat G X$.
Let $\alpha$ be small enough that all $\{g_t\}$-orbits of length $8\alpha$ remain local. Fix $\rho>0$ much smaller than $\alpha$. Fix some large $\tau$ and for each $c \in \hat G\tilde X$ pick an open good geometric rectangle $\tilde R(c,\tau)$ with diameter less than $\rho$. Then $\{ g_{(-\rho,0)} \tilde R(c, \tau)\}_{c\in G\tilde X}$ is an open cover of $\hat G\tilde X$. By compactness of $\hat GX$, we can choose a finite set $\{ \tilde H_1, \ldots , \tilde H_n\}$, writing $\tilde H_i = g_{(-\rho,0)}\tilde R_i$, so that $\hat GX$ is covered by the projections $H_i= g_{(-\rho,0)} R_i$ of $\tilde H_i$ to $GX$. We build our proper family recursively.  Let $B_1 \subset D_1$ be a closed good geometric rectangle of diameter less than $\alpha$ chosen so that $R_1 \subset  \Int_g  B_1$. Note that $H_1 \subset g_{(-\alpha,0)} \Int_g B_1$.

Now suppose that $\{ (B_j, D_j) \}_{j=1}^l$ have been chosen satisfying $\diam D_j < \alpha$, $D_j \cap D_k = \emptyset$ for $j\neq k$, and so that each $D_j$ has the form $g_{s_j}R_i$ for some $s_j$ with $|s_i| << \alpha$ and some good geometric rectangle $R_i$. Let $H_i$ be the element of our cover of smallest index such that $H_i \nsubset \bigcup_{j=1}^l g_{[-\alpha,0]}  \Int_g  B_j$.
We want to build further $(B_j, D_j)$ covering $H_i$. Let 
\[
	M_i = \left\{ c\in R_i : g_{(0,\rho)} c \ \cap \ \big(\bigcup_{j=1}^l  \Int_g  B_j\big) = \emptyset \right\} 
		 = R_i \setminus \left( \bigcup_{j=1}^l g_{(-\rho,0)}  \Int_g  B_j \right). 
\]
$M_i$ is a closed subset of $R_i$. Pick $\epsilon << \frac{\rho}{4Ll}$, where $L$ is given by Lemma \ref{lem:X GX bound}. By passing to endpoints of its geodesics, $M_i$ can be identified with a closed subset of $U^-\times U^+$, so we can find a finite set $T_1, \ldots , T_n$ of closed rectangles with each $T_k$ identified with some $V_k^-\times V_k^+ \subset U^-\times U^+$ such that $\{ \Int_g  T_k\}_{k=1}^n$ cover $M_i$,  $T_k \cap M_i \neq \emptyset$, and $\diam T_k < \epsilon$.

By Lemma \ref{lem:deform tool}, if for some $t\in [0,\rho]$,  $g_tT_k \cap D_j \neq \emptyset$, then for $|t'-t|>2L\epsilon$, $g_{t'}T_k$ and $D_j$ are disjoint. Since $4L\epsilon << \frac{\rho}{l}$, and since there are at most $l$ of the $D_j$'s which can intersect $g_{[0,\rho]}T_k$, we can pick distinct $s_k \in (0, \rho)$ so that $g_{s_k} T_k \cap ( \bigcup _{j=1}^l D_j ) = \emptyset$.

We add the collection $\{ D_k:=g_{s_k}T^k \}_{k=1}^n$ to our collection $\{D_j\}$. Inside each new $D_k$, we choose a slightly smaller closed rectangle $B_k$ so that $\{ g_{-s_k}  \Int_g B_k \}$ cover $M_i$. It is then clear since $\rho<\alpha$ that $\bigcup_{j=1}^{l'} g_{(-\alpha,0)} \Int_g  B_j$ covers $H_i$.

We continue this way until $\hat GX$ is covered by $\{ g_{(-\alpha,0)}  \Int_g B_j \}$ and check the conditions of Definitions \ref{defn:proper} and \ref{defn:pollicott}. We have ensured that \ref{defn:proper}(2) is satisfied. Using the Lipschitz property of the flow and the fact that $\epsilon << \alpha$ we can ensure that $\diam D_j < \alpha$ for all $j$, ensuring condition \ref{defn:proper}(1).  We have also ensured \ref{defn:proper}(3) by constructing the $D_j$ disjoint and picking $\alpha$ so small that all orbit segments with length $8\alpha$ are local.  Applying Lemma \ref{markovpf} produces a pre-Markov proper family satisfying Definition \ref{defn:pollicott}. By construction, each $D_i$ in $\mathcal{D}$ is the image of a good geometric rectangle under the flow for a small time.
\end{proof}

We now complete the proof of Theorem \ref{t.main}.  The flow is a metric Anosov flow by Theorem \ref{t.smale}. The flow is H\"older by Lemma \ref{lem:Lipschitz}. We take a pre-Markov proper family for the flow for which the family of sections $\DDD$ are good geometric rectangles flowed for some short constant amount of time, as provided by Proposition \ref{prop:good mpf}. By Propositions \ref{prop:Lipschitz} and \ref{projholder} the return time map and projection map to these sections are H\"older. Thus, we have met the hypotheses of Theorem \ref{t.usefulversion} and we conclude that the geodesic flow has a strong Markov coding.

%

\section{Projective Anosov representations} \label{sec:twoexamples}

We show that the methods introduced in the previous section can be adapted to the geodesic flow $(\mathsf{U}_\rho\Gamma, (\phi_t))$ for a projective Anosov representation $\rho: \Gamma \to \mathsf{SL}_m(\mathbb{R})$, proving Theorem \ref{t.main2}. This flow is a H\"older continuous topologically transitive metric Anosov flow \cite[Proposition 5.1]{bcls}, so to meet the hypotheses of Theorem \ref{t.usefulversion} it remains to show there is a pre-Markov proper family of sections to the flow such that the return time function between any two sections is H\"older, and the projection from a flow neighborhood of a section to the section are H\"older.  We sketch the proof by showing how to set up analogues of all the objects defined in \S \ref{sec:good rectangles}. This will demonstrate that the proof in \S \ref{sec:good rectangles} applies in this setting.

Following \cite{bcls}, we define the geodesic flow for a projective Anosov representation. Let $\Gamma$ be a Gromov hyperbolic group. We write $\widetilde{\mathsf{U}_0\Gamma} = \partial_\infty \Gamma^{(2)} \times \mathbb R$, and $\mathsf{U}_0\Gamma$ for the quotient $\widetilde{\mathsf{U}_0\Gamma}/\Gamma$. The Gromov geodesic flow (see \cite{champetier} and \cite{mineyev}) can be identified with the $\mathbb R$-action on $\mathsf{U}_0\Gamma$.
\begin{defn}
A representation $\rho:\Gamma\to\mathsf{SL}_m(\mathbb{R})$ is a \emph{projective Anosov representation} if:
\begin{itemize}
	\item $\rho$ has transverse projective limit maps. That is, there exist $\rho$-equivariant, continuous maps $\xi:\partial^\infty\Gamma \to \mathbb{RP}(m)$ and $\theta:\partial^\infty\Gamma \to \mathbb{RP}(m)^*$ such that if $x\neq y$, then
	\[ \xi(x) \oplus \theta(y) = \mathbb{R}^m.\]
	Here we have identified $\mathbb{RP}(m)^*$ with the Grassmannian of $m-1$-planes in $\mathbb{R}^m$ by identifying $v\in\mathbb{RP}(m)^*$ with its kernel.

	\item	We have the following contraction property (see \S2.1 of \cite{bcls}). Let $E_\rho = \widetilde{\mathsf{U}_0\Gamma} \times \mathbb R^{m}/ \Gamma$ be the flat bundle associated to $\rho$ over the geodesic flow for the word hyperbolic group on $\mathsf{U}_0\Gamma$, and let $E_\rho = \Xi \oplus \Theta$ be the splitting induced by the transverse projective limit maps $\xi$ and $\theta$. Let $\{\tilde \psi_t\}$ be the flow on $\widetilde{\mathsf{U}_0\Gamma} \times \mathbb R^{m}$ obtained by lifting the Gromov geodesic flow on $\mathsf{U}_0\Gamma$ and acting trivially on the $\mathbb R^m$ factor. This flow descends to a flow $\{\psi_t\}$ on $E_\rho$. We ask that there exists $t_0>0$ such that for all  $Z\in \mathsf{U}_0\Gamma$, $v \in \Xi_Z\setminus\{0\}$ and $w \in \Theta_Z \setminus \{ 0\}$, we have
\[ \frac{\| \psi_{t_0} (v) \|}{\| \psi_{t_0} (w) \|} \leq \frac{1}{2} \frac{ \| v \|}{\| w \|}. \]	
\end{itemize}
\end{defn}

For $v\in (\mathbb{R}^m)^*$ and $u\in\mathbb{R}^m$, we write $\langle v | u \rangle$ for $v(u)$. We define the geodesic flow $(\mathsf{U}_\rho\Gamma,(\phi_t))$ of a projective Anosov representation, referring to \S4 of \cite{bcls} for further details. Let
\[ F_\rho = \left\{ (x,y,(u,v)) : (x,y)\in\partial^\infty\Gamma^{(2)}, \ u\in \xi(x), \ v\in \theta(y), \ \langle v|u\rangle =1\right\}/\sim\]
where $(u,v)\sim(-u,-v)$ and $\partial^\infty\Gamma^{(2)}$ denotes the set of distinct pairs of points in $\partial^\infty\Gamma$. Since $u$ determines $v$, $F_\rho$ is an $\mathbb{R}$-bundle over $\partial^\infty\Gamma^{(2)}$. The flow is given by
\[ \phi_t(x,y,(u,v)) = (x,y,(e^tu,e^{-t}v)).\]
We define $\mathsf{U}_\rho\Gamma = F_\rho/\Gamma$. The space $\mathsf{U}_\rho\Gamma$ is compact \cite[Proposition 4.1]{bcls} (even though $\Gamma$ does not need to be the fundamental group of a closed manifold). The flow $(\phi_t)$ descends to a flow on $\mathsf{U}_\rho\Gamma$. The flow $(\mathsf{U}_\rho\Gamma, (\phi_t))$ is what we call the \emph{geodesic flow of the projective Anosov representation}.  The flow is H\"older orbit equivalent to the Gromov geodesic flow on $\mathsf{U}_0\Gamma$, which motivates this terminology. In \cite[Theorem 1.10]{bcls}, it is proven that $(\mathsf{U}_\rho\Gamma, (\phi_t))$ is metric Anosov. We construct sections locally on $F_\rho$ and project the resulting sections down to $\mathsf{U}_\rho\Gamma$ that will verify the hypotheses of Theorem \ref{t.usefulversion}, and thus show that the geodesic flow has a strong Markov coding.

We can define stable and unstable foliations in the space $F_\rho$.  For a point $Z=(x_0,y_0,(u_0,v_0))\in F_\rho$, we define respectively, the strong unstable, unstable, strong stable, and stable leafs through $Z$ as follows.
\begin{align*}
 W^{uu}(Z) &=\{ (x,y_0,(u,v_0)) : x\in\partial^\infty\Gamma, x\neq y_0, u\in\xi(x), \langle v_0|u \rangle =1\}. \\
  W^u(Z) &=\{ (x,y_0,(u,v)) : x\in\partial^\infty\Gamma, x\neq y_0, u\in\xi(x), v\in\theta(y_0), \langle v|u \rangle =1\} \\
 & = \bigcup_{t\in\mathbb{R}} \phi_t(W^{uu}(Z)). \\
 W^{ss}(Z) &=\{ (x_0,y,(u_0,v)) : y\in\partial^\infty\Gamma, x_0\neq y, v\in\theta(y), \langle v|u_0 \rangle =1\}. \\
 W^s (Z) &=\{ (x_0,y,(u,v)) : y\in\partial^\infty\Gamma, x_0\neq y, u\in\xi(x_0), v\in\theta(y), \langle v|u \rangle =1\} \\
&  = \bigcup_{t\in\mathbb{R}} \phi_t(W^{ss}(Z)).
\end{align*}

Fix any Euclidean metric $|\cdot|$ on $\mathbb{R}^m$. This induces a metric on 
\[ \mathbb{RP}(m) \times \mathbb{RP}(m)^* \times ((\mathbb{R}^m \times (\mathbb{R}^m)^*)/\pm1).\]
Let $d_{F_\rho}$ be the pull-back of this metric to $F_\rho$; the transversality condition on the limit maps in the definition of Anosov projective representation ensures this is well-defined. This is called a \emph{linear metric} on $F_\rho$. There is a $\Gamma$-invariant metric $d_0$ on $F_\rho$ which is locally bi-Lipschitz to any linear metric by \cite[Lemma 5.2]{bcls}. Therefore, it is sufficient to verify the H\"older properties we want with respect to a linear metric.

We now build our sections in analogy with our construction of good geometric rectangles in the $\CAT(-1)$ setting. Fix some $Z = (x_0,y_0,(u_0,v_0)) \in F_\rho$ and choose some small, disjoint open sets $U^+$ containing $x_0$ and $U^-$ containing $y_0$. Choose $U^+$ and $U^-$ small enough that for all $x,y\in U^+\times U^-$, $\xi(x)$ and $\theta(y)$ are transversal. Since $\xi(x_0)$ and $\theta(y_0)$ are transversal and $\xi, \theta$ are continuous, this is possible. 
Let
\[ R(Z,U^+,U^-) = \{ (x,y,(u,v))\in F_\rho : x\in U^+, y\in U^-, \langle v_0 | u \rangle =1\}.\]

It is straightforward to check that $R(Z,U^+,U^-)$ is a transversal to the flow $\phi_t$ by using the definition of a linear metric to verify that all points sufficiently near to $Z$ project to $R(Z,U^+,U^-)$. It is also straightforward to check that $R(Z,U^+,U^-)$ is a rectangle using the definitions of the (strong) stable and unstable leaves. This is essentially the same as our proof of Lemma \ref{sectionproof}. We can describe $R(Z,U^+,U^-)$ as the zero set for a `Busemann function' as follows.  

\begin{lem}\label{lem:beta Lipschitz}
Fiz $Z_0=(x_0,y_0,(u_0,v_0))$. For all $(x,y,(u,v))$ define
\[ \beta_{Z_0}((x,y,(u,v))) = -\log\langle v_0 | u\rangle.\]
Then $\beta_{Z_0}$ is a locally Lipschitz function with respect to a linear metric on $F_\rho$.
\end{lem}

\begin{proof}
Let $Z_1=(x_1,y_1,(u_1,v_1))$ and $Z_2=(x_2,y_2,(u_2,v_2))$ be in a small neighborhood of $Z_0$ for the linear metric. This implies that $\langle v_0,u_i\rangle$ lie in some range bounded away from zero. Over this range, the function $-\log$ is Lipschitz.

We know by the definition of a linear metric that 
\[ d_{F_\rho}(Z_1,Z_2) = |\xi(x_1)-\xi(x_2)| + |\theta(y_1)-\theta(y_2)| + |u_1-u_2| + |v_1-v_2|. \]
(In the various factors above, $|* - *|$ denotes the metrics induced on $\mathbb{RP}(m)$, $\mathbb{RP}(m)^*$, $\mathbb{R}^m$, and $(\mathbb{R}^m)^*$ by the Euclidean metric on $\mathbb{R}^m$.)
We calculate, using that $-\log$ and $\langle v_0 | \cdot\rangle$ are Lipschitz:
\begin{align}
	|\beta_{Z_0}(Z_1)-\beta_{Z_0}(Z_2)| = |-\log\langle v_0|u_2\rangle - \log\langle v_0|u_1\rangle | \nonumber 
					 & \leq K_1 |\langle v_0|u_2 \rangle - \langle v_0|u_1 \rangle| \nonumber \\
					& \leq K_2 |u_2-u_1| \nonumber \\
					& \leq K_2 d_{F_\rho}(Z_1,Z_2). \nonumber  \qedhere
\end{align}
\end{proof}

\begin{lem}
For all $Z\in R(Z_0,U^+,U^-)$, we have $\beta_{Z_0}(\phi_t Z) = -t$.
\end{lem}

\begin{proof}
This is immediate from the definition of $\beta_{Z_0}$.
\end{proof}

It is clear that $R(Z_0, U^+, U^-) = \{ (x,y,(u,v)) : x\in U^+, y\in U^-, \beta_{Z_0}(u)=0\}$ and if $\phi_{t^*}Z \in R(Z_0,U^+,U^-)$, then $\beta_{Z_0}(Z)=t^*$. We now have a simple proof of the analogue of Proposition \ref{prop:Lipschitz} we need:

\begin{prop}
The return time function between two good geometric rectangles is Lipschitz.
\end{prop}

\begin{proof}
Suppose that $Z_1, Z_2 \in R$ and, for small $r_1, r_2$, that $\phi_{r_1}Z_1, \phi_{r_2}Z_2 \in R'=R(Z',U'^+, U'^-)$. 
Then by Lemma \ref{lem:beta Lipschitz}, we have
\[
|r_1-r_2|  = |\beta_{Z'}(Z_1) - \beta_{Z'}(Z_2)| \leq K d_{F_\rho}(Z_1, Z_2), 
\]
and we thus conclude that the return time function from $Z_1$ to $Z_2$ is Lipschitz.
\end{proof}

It is also easy to verify that the flow $(\phi_t)$ is Lipschitz. All that is left to prove is an analogue of Proposition \ref{projholder}:

\begin{lem}
For any good geometric rectangle $R$, $\Proj_R:\phi_{(-\alpha,\alpha)} R \to R$ is H\"older.
\end{lem}

\begin{proof}
Since the flow is Lipschitz, we can assume $Z_1\in R$.  Assume $Z_2 \in \phi_{-t^*}R$ for some $t^* \in (-\alpha, \alpha)$, so $\Proj_R(Z_2) = \phi_{t^*}Z_2$. If $R$ is a rectangle based at $Z_0=(x_0,y_0,(u_0,v_0))$, then $(u_2, v_2)\mapsto (e^{t^*}u_2, e^{-t^*}v_2)$ is the projection along the smooth flow $(e^t, e^{-t})$ to the smooth subset of $\mathbb{R}^m\times (\mathbb{R}^m)^*$ given by $\{(u,v): \langle v_0 | u \rangle =1, \langle v|u \rangle =1\}$, which is transverse to the flow. Therefore this map is smooth, hence Lipschitz on any compact set for any linear metric, and this suffices for the proof.
\end{proof}

%

\section{Applications of strong markov coding}\label{s.applications}

There is a wealth of literature for Anosov and Axiom A flows which uses the strong Markov coding to prove strong dynamical properties of equilibrium states. We do not attempt to create an exhaustive list of these applications, but we refer the reader to the many results described in references such as Bowen-Ruelle \cite{bowen-ruelle}, Pollicott \cite{pollicott}, Denker-Philipp \cite{DP84} and Melbourne-T\"or\"ok \cite{MT04}. 

We summarize some of these applications as they apply to the non-wandering set of the geodesic flow of a convex cocompact locally CAT(-1) space $X = \tilde X/\Gamma$. The flow is topologically transitive since the action of $\Gamma$ on $(\Lambda \times \Lambda)\setminus \Delta$ is topologically transitive. Thus, the shift of finite type in the strong Markov coding is irreducible. In places in the discussion below, we need the notion of topological weak-mixing. We say that a metric Anosov flow is \emph{topologically weak-mixing} if all closed orbit periods are not integer multiples of a common constant. 

The result that there is a unique equilibrium state $\mu_\varphi$ for every H\"older potential is due to Bowen-Ruelle \cite{bowen-ruelle} for topologically transitive Axiom A flows. The method of proof was observed to extend to flows with strong Markov coding in \cite{pollicott}. It is also observed in \cite{pollicott} that if $\varphi, \psi$ are H\"older continuous functions then the map $t \to P(\varphi + t \psi)$ is analytic and $(d/dt)P(\varphi+ t \psi)|_{t=0} = \int \psi d_{\mu_\varphi}$, where $P(\cdot)$ is the topological pressure. This result is one of the key applications of thermodynamic formalism used in \cite{bcls, aS16}.

We now discuss the statistical properties listed in (1) of Corollary D. The  Almost Sure Invariance Principle (ASIP), Central Limit Theorem (CLT), and Law of the Iterated Logarithm are all properties of a measure that are preserved by the push forward $\pi^\ast$ provided by the strong Markov coding, and thus it suffices to establish them on the suspension flow. The CLT is probably the best known of these results, and goes back to Ratner \cite{ratner}. 
A convenient way to obtain these results in our setting is to apply the paper of Melbourne and T\"or\"ok \cite{MT04} which gives a relatively simple argument that the CLT lifts from an ergodic measure in the base to the corresponding measure on the suspension flow. They then carry out the more difficult proof that the ASIP lifts from an ergodic measure in the base to the flow, recovering the result of Denker and Phillip \cite{DP84}. The other properties discussed (and more, see \cite{MT04}), are a corollary of ASIP. The equilibrium state for the suspension flow is the lift of a Gibbs measure on a Markov shift. The measure in the base therefore satisfies ASIP by \cite{DP84}, so we are done. 

We now discuss the application to dynamical zeta functions, which was established in the case there is a strong Markov coding and the flow is topologically weak mixing in \cite{pollicott}.   Results on zeta functions are carried over from the suspension flow by a strong Markov coding. The assumption of topological weak mixing is not needed for the result that we stated as (2) in Corollary D. See \cite[Chapter 6]{PP}.

For item (3) of Corollary D, we can refer directly to \cite{pollicott} for the statement that if the flow has a strong Markov coding and is  topological weak-mixing, then the equilibrium state $\mu_\varphi$ is Bernoulli. The proof is given by Ratner \cite{ratner2}.

For item (4) of Corollary \ref{t.applications} , we argue as follows. Ricks proves that for a proper, geodesically complete, $\CAT(0)$ space $\tilde X$ with a properly discontinuous, cocompact action by isometries $\Gamma$, all closed geodesics have lengths in $c\mathbb{Z}$ for some $c>0$ if and only if $\tilde X$ is a tree with all edge lengths in $c\mathbb{Z}$ \cite[Thorem 4]{ricks}. It follows that $X$ is a metric graph with all edges of length $c$. In this case, the symbolic coding for the geodesic flow on $X$ is explicit: $(GX, (g_t))$ is conjugate to the suspension flow with constant roof function $c$ over the subshift of finite type defined by the adjacency matrix $A$ for the graph $X$.  Equilibrium states for the flow are products of equilibrium states in the base with Lebesgue measure in the flow direction. Since an equilibrium state for a H\"older potential on a topologically mixing shift of finite type is Bernoulli, item (4) follows immediately by taking $k\geq 1$ so that $A^k$ is aperiodic; if $k=1$, the measure on the base is Bernoulli, and if $k>1$ the measure on the base is the product of Bernoulli measure and rotation of a finite set with $k$ elements. 

\subsection*{Acknowledgements}
D.C. would like to thank Ohio State University for hosting him for several visits, during which much of this work was done. We are grateful to Richard Canary for several discussions regarding the application of our work to the results of \cite{bcls}.

%

\bibliographystyle{alpha}
\bibliography{biblio}

\end{document}